\documentclass[preprint,times]{elsarticle}
\usepackage{lineno}

\journal{Elsevier}

\usepackage{stmaryrd}
\usepackage{amsmath}
\usepackage{amssymb}
\usepackage{amsthm}
\usepackage{graphicx}
\usepackage{epic,eepic,epsfig}
\usepackage{color}
\usepackage{subfigure}  
\usepackage{placeins}   
\usepackage{multirow}
 
\usepackage{varwidth}
\usepackage{float}

\usepackage{hyperref} 
\hypersetup{
	colorlinks=true, 
	urlcolor=black, 
	linkcolor=black, 
	citecolor=black,
}

\usepackage{algorithm,algorithmic}
\makeatletter
\newenvironment{breakablealgorithm}
{
		\begin{center}
			\refstepcounter{algorithm}
			\hrule height.8pt depth0pt \kern2pt
			\renewcommand{\caption}[2][\relax]{
				{\raggedright\textbf{\ALG@name~\thealgorithm} ##2\par}%
				\ifx\relax##1\relax 
				\addcontentsline{loa}{algorithm}{\protect\numberline{\thealgorithm}##2}%
				\else 
				\addcontentsline{loa}{algorithm}{\protect\numberline{\thealgorithm}##1}%
				\fi
				\kern2pt\hrule\kern2pt
			}
		}{
		\kern2pt\hrule\relax
	\end{center}
}
\makeatother

\newtheorem{lemma}{Lemma}[section]
\newtheorem{theorem}{Theorem}[section]

\newtheorem{example}{Example}[section] 
\newtheorem{remark}{Remark}[section]
\newtheorem{corollary}{Corollary}[section]

\numberwithin{algorithm}{section}
\numberwithin{equation}{section}

\usepackage{lineno} 

\allowdisplaybreaks[4]

\newcommand{\rel}{\mathrm{rel}}

\newcommand{\est}{\mathrm{est}}
\newcommand{\rem}{\mathrm{rem}}
\newcommand{\reduce}{\mathrm{red}}
\newcommand{\ctr}{\mathrm{ctr}}
\newcommand{\conv}{\mathrm{conv}}
\newcommand{\qoo}{\mathrm{qo}}

\newcommand{\anorm}[1]{\|#1\|_a}
\newcommand{\veps}{\varepsilon}
\newcommand{\vect}[1]{\boldsymbol{#1}}
\newcommand{\vn}{\boldsymbol{n}}
\newcommand{\vx}{\boldsymbol{x}}
\newcommand{\valpha}{\boldsymbol{\alpha}}
\newcommand{\vbeta}{\boldsymbol{\beta}}
\newcommand{\ha}{\hat{a}}
\newcommand{\hu}{\hat{u}}
\newcommand{\hmcP}{\hat{\mathcal{P}}}
\newcommand{\hmcL}{\hat{\mathcal{L}}}
\newcommand{\homega}{\hat{\omega}}

\newcommand{\tu}{\tilde{u}}
\newcommand{\wtC}{\widetilde{C}}
 
\newcommand{\mcE}{\mathcal{E}}

\newcommand{\mcL}{\mathcal{L}}
\newcommand{\mcM}{\mathcal{M}}
\newcommand{\mcN}{\mathcal{N}}
\newcommand{\mcT}{\mathcal{T}}
 
\newcommand{\mbI}{\mathbb{I}}
\newcommand{\mbR}{\mathbb{R}}
\newcommand{\mbV}{\mathbb{V}}
 
\newcommand{\Ltw}[1]{\|#1\|_0}
\newcommand{\Hon}[1]{\|#1\|_1}

\newcommand{\LtwD}[2]{\|#1\|_{0,#2}}

\newcommand{\LtwT}[1]{\|#1\|_{0,T}}

\newcommand{\LtwE}[1]{\|#1\|_{0,E}}
\newcommand{\mrd}{\mathrm{d}}

\def\leq{\leqslant}
\def\geq{\geqslant}

\DeclareMathOperator*{\esssup}{ess\,sup}

\begin{document}

\begin{frontmatter}
	
\title{A correction adaptive two-grid finite element method for nonselfadjoint or indefinite elliptic problems}

\author[LNSF]{Fei Li}
\ead{lifei@lingnan.edu.cn}

\author[MST]{Qingguo Hong}
\ead{qingguohong@mst.edu}

\author[GZU]{Ming Tang}
\ead{tangming@gzhu.edu.cn}

\author[SCNU]{Liuqiang Zhong\corref{cor}}
\ead{zhong@scnu.edu.cn}

\cortext[cor]{Corresponding author}
\address[LNSF]{School of Mathematics and Statistics, Lingnan Normal University, Zhanjiang, Guangdong, 524048, China}
\address[MST]{Department of Mathematics and Statistics, Missouri University of Science and Technology, Rolla, MO 65409, USA}
\address[GZU]{School of Mathematics and Information Science, Guangzhou University, Guangzhou, Guangdong, 510405, China}
\address[SCNU]{School of Mathematical Sciences, South China Normal University, Guangzhou, Guangdong, 510631, China}

\begin{abstract}
We propose, analyze, and numerically validate a correction adaptive two-grid finite element method (CATGFEM) for nonselfadjoint or indefinite elliptic problems. In contrast to the adaptive two-grid finite element method (ATGFEM) of Li and Zhang [SIAM J. Sci. Comput., 43 (2021), pp. A908–A928], which is restricted to symmetric positive-definite problems, the proposed method introduces an additional correction step that solves a small-scale discrete residual problem on the coarse mesh. This step entails negligible additional computational cost and allows us to show that the $L^2$-norm error of the corrected discrete solution is a higher-order of the energy-norm error of the discrete solution. Using this result, we prove a contraction property for a suitable sum of quasi-errors on two successive adaptive meshes and establish convergence of the method. Numerical experiments illustrate the improved effectiveness and robustness of our method in comparison with ATGFEM.
 \end{abstract}


\end{frontmatter}


\section{Introduction}\label{sect1}
Nonselfadjoint or indefinite elliptic problems constitute a fundamental class of model problems in computational science and engineering. Representative examples include convection–diffusion–reaction equations \cite{DuSHXieXP15:1327,ZhongLQXuanY22:113903,AbdulleSouza22:110894} and the Helmholtz equation \cite{DuanSYWuHJ23:21,IrimieBouillard01:4027,DuYWuHJ15:782}. In many applications, the solutions of such problems may exhibit singular behavior caused by nontrivial domain geometries or nonsmooth coefficients and source terms \cite{Grisvard92Book,Mitchell13:350}. The adaptive finite element method (AFEM) \cite{Verfurth96book,verfurth13Book,AinsworthOden00Book} is one of the most effective approaches for treating elliptic problems with solution singularities. AFEM constructs computable a posteriori error estimators based on the discrete solution and known problem data, identifies elements with large local error contributions associated with singular behavior, and locally refines these elements to generate improved meshes. Repeating this adaptive loop balances discretization errors between singular and nonsingular regions, thereby yielding quasi-optimal meshes and accurate numerical solutions.

In many existing analyses of AFEM for nonselfadjoint or indefinite elliptic problems, see, for example, \cite{NochettoSiebert09:409,BespalovHaberl17:318,FeischlFuhrer14:601,MekchayNochetto05:1803}, the availability of the exact discrete solution is assumed, although its computation is in general difficult \cite{XuJC92:303}. It is well known that symmetric positive-definite (SPD) problems are typically easier to solve than nonselfadjoint or indefinite problems, and two-grid methods provide a classical strategy for transforming the latter into the former \cite{XuJC96:1759}.
Li and Zhang \cite{LiYkZhangY21:A908} incorporated the two-grid idea into adaptivity and proposed the adaptive two-grid finite element method (ATGFEM). In this approach, the 
$(k-1)$th and $k$th adaptive meshes are regarded as the coarse and fine grids, respectively, in contrast to the method proposed in \cite{BiCJWangC18:23}, which additionally introduces a uniform coarse grid at each adaptive level. Using the discrete solution on the $(k-1)$th mesh, ATGFEM constructs and solves only an SPD problem on the 
$k$th mesh.

It should be noted that the convergence result for ATGFEM in \cite{LiYkZhangY21:A908} holds up to certain $L^2$-norm errors of the discrete solutions, which are empirically treated as higher-order quantities of the corresponding energy-norm errors, under the assumption that the initial mesh is sufficiently fine. However, this key observation is not theoretically justified in \cite{LiYkZhangY21:A908}, and therefore the convergence analysis is not fully rigorous. Moreover, our numerical investigations indicate that the performance of ATGFEM may deteriorate in certain regimes, such as problems with small negative reaction coefficients or when large marking parameters are used.

For nonselfadjoint or indefinite elliptic problems, the main differences and contributions of the present work, in comparison with \cite{LiYkZhangY21:A908}, are summarized as follows:
\begin{enumerate}
\item We introduce an additional correction step by solving a discrete residual problem on the initial mesh, whose solution is used to correct the discrete solution obtained on the previous adaptive mesh. Consequently, starting from the second adaptive step, the proposed correction adaptive two-grid finite element method (CATGFEM, Algorithm~\ref{alg-AFEMC}) behaves as an iterative two-grid scheme, with the fixed initial mesh serving as the coarse grid and the adaptively refined mesh as the fine grid.
\item We rigorously prove that a suitable sum of quasi-errors on two successive adaptive levels is contractive (Theorem~\ref{thm-contract}), which yields convergence of the CATGFEM (Corollary~\ref{cor-convergence}). The contraction proof relies crucially on a new auxiliary result showing that the $L^2$-norm error of the corrected discrete solution is a higher-order-order of the energy-norm error of the discrete solution (Lemma~\ref{lem-L2-lifting}).
\item Numerical experiments confirm the theoretical results and demonstrate that the proposed CATGFEM exhibits improved effectiveness and greater robustness with respect to the marking parameter when compared with ATGFEM.
\end{enumerate}
The remainder of this paper is organized as follows. Section~\ref{part-preliminary} introduces preliminaries, and Section~\ref{part-algorithm} presents the CATGFEM. Section~\ref{sect3} establishes the key $L^2$-error estimate and develops the contraction and convergence analysis. Numerical experiments illustrating the effectiveness and robustness of the proposed method are reported in Section~\ref{sect-test}.

\section{Preliminaries}\label{part-preliminary}
In this section, we introduce some notations, the nonselfadjoint or indefinite elliptic model problem, and the finite element discretization. 

\subsection{Notation}
Let $\Omega \subset \mbR^d \; (d = 2, \, 3)$ be a simply connected bounded open polygonal or polyhedral domain with Lipschitz boundary $\partial \Omega$. For any set $S \subset \Omega$ and any $1 \leq p \leq \infty$, 
we denote by $L^p(S)$ the standard Lebesgue space equipped with the norm 
\begin{align*}
	\|v\|_{0,p,S} := \left(\int_S |v|^p \mrd \vx \right)^\frac{1}{p}, \; p \in (1,\infty), \quad \text{or} \quad \|v\|_{0,\infty,S}:=\esssup\limits_{\vx \in S}|v(\vx)|. 
\end{align*}
Specially, we write $\|v\|_{0,2,S}$ to be $\|v\|_{0,S}$ for any $v \in L^2(S)$. We also denote the $L^2(S)$ inner product by 
\begin{equation*}
	(v,w)_S := \int_S vw \mrd \vx, \quad \forall \; v,w \in L^2(S). 
\end{equation*}

For any integer $m > 0$, we denote by $W^{m,p}(S)$ the standard Sobolev space equipped with the norm 
\begin{equation*}
	\|v\|_{m,p,S} := \left(\sum_{|j| \leq m}\|D^j v\|^p_{0,p,S}\right)^\frac{1}{p}, \; p \in (1,\infty), \quad \text{or} \quad \|v\|_{m,\infty,S}:=\max_{|j| \leq m}\|D^j v\|_{0,\infty,S}, 
\end{equation*}
where $D^j v$ is the $j$-th weak partial derivative. 
In particular, we write $H^m(S)$  instead of $W^{m,2}(S)$ and write  $\|v\|_{m,S}$  instead of $\|v\|_{m,2,S}$, respectively. Moreover, note that if $S = \Omega$, the subscript $S$ will be dropped in the norm and the inner product. 

The space $H_0^1(\Omega)$ consists of all functions $v \in H^1(\Omega)$ with zero traces, and $H^{-1}(\Omega)$ denotes the dual space of $H_0^1(\Omega)$ with the norm 
	\begin{equation}\label{def-negative-norm}
		\|g\|_{-1} := \sup_{v \in H_0^1(\Omega)}\frac{\langle g,v \rangle}{\|v\|_1}, \quad \forall \; g \in H^{-1}(\Omega), 
	\end{equation}
where $\langle \cdot,\cdot \rangle$ denotes the dual pairing between $H^{-1}(\Omega)$ and $H_0^1(\Omega)$.

The inequality $a\lesssim b$ abbreviates $a\leq Cb$ with some generic constant $C > 0$ independent of the notations $a$ and $b$, the mesh size, and the adaptivity iteration counter mentioned later. Furthermore, $a \simeq b$ denotes $a \lesssim b$ and $b \lesssim a$. We also use $C = C(a,b)$ to indicate that the constant $C$ depends only on the notations $a$ and $b$.

\subsection{Model problem}
Given the source term $f \in H^{-1}(\Omega)$, we consider the following nonselfadjoint or indefinite elliptic boundary-value problem
\begin{alignat}{4}
	\hat{\mcL}u:=-\nabla \cdot(\boldsymbol{\alpha}\nabla u)+\boldsymbol{\beta} \cdot \nabla u+\gamma u  =~&f,  &\quad&\text{in}~\Omega,  \label{model-1} \\
	u  =~&0,&\quad&\text{on}~ \partial \Omega,\label{model-2} 
\end{alignat}
where the matrix function (also called the diffusion coefficient) $\vect{\alpha} \in W^{1,\infty}(\Omega)^{d\times d}$ is SPD with eigenvalues $0 < \mu_0 < \mu_1 < \infty$ such that
	\begin{equation}\label{ellipticity}
		\mu_0|\boldsymbol{\xi}|^2 \leq \boldsymbol{\xi}\valpha(\vx)\boldsymbol{\xi}^T  \leq \mu_1|\boldsymbol{\xi}|^2, \quad \forall \; \boldsymbol{\xi} \in \mathbb{R}^d, ~ \vx \in \Omega, 
	\end{equation}
the vector function (also called the convection coefficient) $\vect{\beta} \in W^{1,\infty}(\Omega)^d$ and the scalar function (also called the reaction coefficient) $\gamma \in L^\infty(\Omega)$. 

The weak formulation of \eqref{model-1}--\eqref{model-2} reads: Find $u \in H_0^1(\Omega)$ such that 
\begin{equation}\label{model-VP}
	\ha(u,v) = \langle f , v \rangle, \quad \forall\; v \in H_0^1(\Omega), 
\end{equation}
where the bilinear form $\ha(\cdot,\cdot)$ is defined by 
\begin{equation}\label{bilinear-ha}
	\ha(v,w) := (\valpha\nabla v,\nabla w) + (\vbeta\cdot \nabla v,w) + (\gamma v,w), \quad
	\forall \; v,w \in H_0^1(\Omega). 
\end{equation}
We suppose that the operator $\hmcL:\; H_0^1(\Omega) \rightarrow H^{-1}(\Omega)$ is an isomorphism, hence the weak formulation \eqref{model-VP} admits a unique solution $u \in H_0^1(\Omega)$.

We define the energy norm as 
	\begin{equation}\label{norm-energy}
		\anorm{v} := \sqrt{a(v,v)}, \quad \text{where} \quad a(v,w) := (\valpha\nabla v,\nabla w), \quad \forall \; v,w \in H_0^1(\Omega). 
	\end{equation}
According to the ellipticity of $\valpha$ in \eqref{ellipticity} and the Poincar\'e inequality, we know that the energy norm $\anorm{\cdot}$ is equivalent to the native norm $\Hon{\cdot}$ on $H_0^1(\Omega)$, i.e., for all $v \in H_0^1(\Omega)$, there holds
\begin{equation}\label{equivalence-norm}
	 \anorm{v} \simeq \Hon{v}. 
\end{equation}
By the assumptions on the model data $\valpha,\,\vbeta$ and $\gamma$, and the norm equivalence \eqref{equivalence-norm}, we obtain the boundedness of the bilinear form $\hat{a}(\cdot,\cdot)$: 
    \begin{equation}\label{ha-boundedness}
		\ha(v,w) \lesssim \anorm{v}\anorm{w}, \quad \forall \; v,w \in H_0^1(\Omega). 
	\end{equation}

Since $\hmcL: \; H_0^1(\Omega) \rightarrow H^{-1}(\Omega)$ is an isomorphism, its dual operator $\hmcL^*: \; H_0^1(\Omega) \rightarrow H^{-1}(\Omega)$ is also an isomorphism, and there both hold the stability estimates (see, e.g., \cite{XuJC96:1759})
	\begin{equation}\label{est-regularity}
		\|v\|_1 \lesssim \|\hmcL v\|_{-1}, \quad \|v\|_1 \lesssim \|\hmcL^* v\|_{-1}, \quad \forall \; v \in H_0^1(\Omega).
	\end{equation}
By the stability estimates \eqref{est-regularity}, the $\|\cdot\|_{-1}$ norm definition \eqref{def-negative-norm} and the norm equivalence \eqref{equivalence-norm}, we obtain 
	\begin{equation}\label{infsup-VP}
		\anorm{v} \lesssim \sup_{w \in H_0^1(\Omega)}\frac{\ha(v,w)}{\anorm{w}}, \quad \anorm{v} \lesssim \sup_{w \in H_0^1(\Omega)}\frac{\ha(w,v)}{\anorm{w}}, \quad \forall \; v \in H_0^1(\Omega). 
	\end{equation}

To discretize the weak formulation \eqref{model-VP}, we consider a conforming and shape-regular simplex partition $\mcT = \{T\}$ of the domain $\Omega$, where the element $T \in \mcT$ is a triangle if $d = 2$, or a tetrahedron if $d = 3$. Associated with the mesh $\mcT$, we consider the continuous piecewise finite element space 
\begin{equation}\label{space-fem}
	\mbV_\mathcal{T}:= \{v_\mathcal{T} \in \mathcal{C}(\Omega): \; v_\mathcal{T}|_T \in \mathbb{P}_\ell(T), \; \text{for any element }  
	T \in \mcT\},  
\end{equation}
where $\mathcal{C}(\Omega)$ is the space of continuous functions on $\Omega$ and $\mathbb{P}_\ell(T)$ consists of all polynomials of total degree less than or equal to $\ell$ on the element $T$. 
The finite element discretization of \eqref{model-VP} reads: Find  $\hu_\mathcal{T} \in \mbV_\mathcal{T}$ such that 
\begin{equation}\label{model-FEM}
	\ha(\hu_\mathcal{T},v_\mathcal{T}) = \langle f,v_\mathcal{T} \rangle, \quad \forall \; v_\mathcal{T} \in \mbV_\mathcal{T}. 
\end{equation}
The existence and uniqueness of the finite element solution of \eqref{model-FEM} can be proved under the condition that the mesh size $h_\mathcal{T}:=\max_{T \in \mcT} h_T$ of the mesh $\mcT$ is sufficiently small \cite{Schatz74:959,SchatzWang96:19}, where $h_T := \mathrm{diam}(T)$ is the diameter of any element $T \in \mcT$. Moreover, note that combining the assumption that the mesh size $h_{\mathcal{T}}$ is sufficiently small with the norm equivalence \eqref{equivalence-norm} yields that  
\begin{equation}\label{infsup-FEM}
	\anorm{v_\mcT} \lesssim \sup_{w_\mcT \in \mbV_\mcT}\frac{\ha(v_\mcT,w_\mcT)}{\anorm{w_\mcT}}, \quad \anorm{v_\mcT} \lesssim \sup_{w_\mcT \in \mbV_\mcT}\frac{\ha(w_\mcT,v_\mcT)}{\anorm{w_\mcT}}, \quad \forall \; v_\mcT \in \mbV_\mcT,  
\end{equation}
one can refer to \cite[Lemma 2.2]{XuJC96:1759}.


\section{Adaptive algorithm}\label{part-algorithm}
The AFEM is based on the loop of the following procedures
\begin{equation}\label{loop}
	\textbf{SOLVE}~\rightarrow~\textbf{ESTIMATE}~\rightarrow~\textbf{MARK}~\rightarrow~\textbf{REFINE}. 
\end{equation}	
We below describe each procedure of loop \eqref{loop} of the proposed CATGFEM (Algorithm \ref{alg-AFEMC}). Particularly in Section \ref{part-solve}, we also detailedly present Procedure \textbf{SOLVE} of the standard AFEM (SAFEM, \cite{NochettoSiebert09:409,FeischlFuhrer14:601}) and the ATGFEM \cite{LiYkZhangY21:A908}, respectively, to emphasize the main idea of our CATGFEM. 

\subsection{Procedure \textbf{SOLVE}}\label{part-solve}
Let $\{\mcT_k\}_{k \geq 0}$ be a mesh family generated by Procedure \textbf{REFINE}; see Section \ref{part-refine} in detail. We write the finite element space $\mbV_{\mcT_k}$ to $\mbV_k$ for simplicity. 

The SAFEM is to solve the discretization \eqref{model-FEM} on each mesh $\mcT_k$: Find $\hu_k \in \mbV_k$ such that 
	\begin{equation}\label{solve-SAFEM}
		\ha(\hu_k,v_k) = \langle f,v_k \rangle, \quad \forall \; v_k \in \mbV_k. 
	\end{equation}
Note that the resulting algebraic system is in general nonselfadjoint or indefinite, and solving such a problem is difficult \cite{XuJC92:303}. Furthermore, numerical computation of \eqref{solve-SAFEM} requires that the initial mesh size $h_0$ is sufficiently small. Thus, these lead to that the computational cost for such a nonselfadjoint or indefinite problem grows expensive as the adaptivity iteration counter $k$ increases. 

It is well-known that solving SPD problems is often much easier than solving nonselfadjoint or indefinite problems, since there are many fast solvers, e.g., multigrid methods \cite{McCormick87Book}, for the resulting algebraic system of the SPD problem. Therefore, the ATGFEM, by applying the two-grid method \cite{XuJC96:1759} to adaptivity, i.e., see the $(k-1)$th and the $k$th meshes as the coarse and fine grids, transforms the nonselfadjoint or indefinite problem into an SPD (approximate) problem to solve. The corresponding solving procedure is stated as follows: 

\begin{enumerate}
	\itemsep = 0pt \parskip = 0pt
	\item On the initial mesh $\mcT_0$, solve the nonselfadjoint or indefinite discretization \eqref{solve-SAFEM} to obtain the finite element solution $\tu_0 := \hu_0$. 
	\item On the mesh $\mcT_k, \; k \geq 1$, solve the following SPD discretization: Find $\tu_k \in \mbV_k$ such that 
	\begin{equation}\label{solve-ATGFEM}
		a(\tu_k,v_k) = \langle f,v_k \rangle - N(\tu_{k-1},v_k), \quad \forall \; v_k \in \mbV_k, 
	\end{equation}
where the bilinear form $a(\cdot,\cdot)$ is defined in \eqref{norm-energy} and the second term in the left-hand side above is called the lower-order term, which is defined by 
\begin{equation}\label{bilinear-form-N}
	N(v,w) := (\vbeta \cdot \nabla v,w) + (\gamma v,w), \quad \forall \; v,w \in H_0^1(\Omega). 
\end{equation}
\end{enumerate}
\begin{remark}\label{rem-small-scale}
Note that the initial mesh $\mcT_0$ in the adaptive mesh family $\{\mcT_k\}_{k \geq 0}$ is the coarsest, hence the nonselfadjoint or indefinite discretization \eqref{solve-SAFEM} on $\mcT_0$ is a relatively small-scale problem in overall adaptivity. 
\end{remark}

However, we find that the ATGFEM behaves not well in some numerical examples; see Section \ref{sect-test} in details. Moreover, the convergence result of the ATGFEM is not rigorous since the analysis implicitly requires the higher-order $L^2$-norm error term which is only empirical but not theoretically justified.

In order to improve the ATGFEM, the proposed CATGFEM incorporates an additional correction step before solving the SPD problem \eqref{solve-ATGFEM}, i.e., first solve a small-scale discrete residual problem on the initial mesh $\mcT_0$ (see Remark \ref{rem-small-scale}), and then use the solution of the residual problem to correct the discrete solution on the previous mesh. The corresponding procedure is described as follows: 
\begin{enumerate}
\itemsep = 0pt \parskip = 0pt
\item On the initial mesh $\mcT_0$, solve the nonselfadjoint or indefinite discretization \eqref{solve-SAFEM} to obtain the finite element solution $u_0 = \hu_0$. 
\item On the mesh $\mcT_k\;(k \geqslant 1)$, perform the following two subprocedures: 
\begin{enumerate}[(1)]
	\itemsep = 0pt \parskip = 0pt
	\item Solve a nonselfadjoint or indefinite residual equation on the initial mesh $\mcT_0$, i.e., find $e_{0,k-1} \in \mbV_0$ such that 
	\begin{equation}\label{mysolver-k-1}
		\ha(e_{0,k-1},v_0) = \langle f,v_0 \rangle - \ha(u_{k-1},v_0), \quad \forall \; v_0 \in \mbV_0.
	\end{equation}
	\item Solve an SPD discretization on $\mcT_k$, i.e., find $u_k \in \mbV_k$ such that
	\begin{equation}\label{mysolver-k-2}
		a(u_k,v_k) = (f,v_k) - N(u_{k-1}+e_{0,k-1},v_k), \quad \forall\; v_k \in \mbV_k. 
	\end{equation}
\end{enumerate}
\end{enumerate}
Note that \eqref{mysolver-k-1} is relatively small-scale; consequently, solving \eqref{mysolver-k-1}--\eqref{mysolver-k-2} makes the treatment of nonselfadjoint or indefinite problems almost as easy as solving SPD problems. 

\begin{remark}\label{rem-CATG-1}For the case of $k = 1$, the right-hand side of \eqref{mysolver-k-1} is zero due to $u_0 = \hu_0$ being the solution of \eqref{solve-SAFEM}. This implies that $e_{0,0} = 0$, which means that \eqref{mysolver-k-1} does not work, and \eqref{mysolver-k-2} will be transformed into 
\begin{equation}\label{solve-CATGFEM-1}
	a(u_1,v_1) = \langle f, v_1 \rangle - N(u_0,v_1), \quad \forall \; v_1 \in \mbV_1, 
\end{equation}
which is indeed \eqref{solve-ATGFEM} for $k=1$ of the ATGFEM due to $u_0 = \tu_0$.  
\end{remark}

\begin{remark}The discretization \eqref{mysolver-k-1} is equivalent to the following form
\begin{equation}\label{mysolver-k-1-equivalence}
	\ha(u_{k-1}^*,v_0) = \langle f, v_0 \rangle, \quad \forall \; v_0 \in \mbV_0, 
\end{equation}
with $u_{k-1}^* := u_{k-1} + e_{0,k-1}$. We note that $u_{k-1}^* \in \mbV_{k-1}$ due to $e_{0,k-1} \in \mbV_0 \subset \mbV_{k-1}$. 
\end{remark}

\begin{remark}\label{rem-solve-difference}
We summarize the similarities and differences of Procedure \textbf{SOLVE} of ATGFEM and our CATGFEM: 
\begin{itemize}
	\itemsep = 0pt \parskip = 0pt
	\item On the initial mesh $\mcT_0$, ATGFEM and our CATGFEM both solve the nonselfadjoint or indefinite discretization \eqref{solve-SAFEM}. 
	\item On the mesh $\mcT_1$, our CATGFEM also behaves like ATGFEM (see Remark \ref{rem-CATG-1}). 
	\item The difference is that, starting from the mesh $\mcT_k,\; k \geq 2$, the SPD problem is constructed by the discrete solution on the previous mesh for the ATGFEM, while by a corrected one for our CATGFEM, where the corrected solution is updated via the solution of a small-scale residual problem \eqref{mysolver-k-1}. 
\end{itemize}
\end{remark}

\subsection{Procedure \textbf{ESTIMATE}}
For the ATGFEM in \cite{LiYkZhangY21:A908}, starting from the mesh $\mcT_k \; (k \geq 1)$, the discrete solutions $\tu_k$ and $\tu_{k-1}$ are used for the SPD operator $\mcL$ and the lower-order term operator $\mcN$, respectively, in the elementwise residual of the error estimator, where 
\begin{equation*}
	\mcL v := -\nabla \cdot (\valpha \nabla v), \quad \mcN v:= \vbeta \cdot \nabla v + \gamma v. 
\end{equation*}
We will adapt this idea for our CATGFEM setting. 

With the discrete solution $u_k$ and $u_{k-1}^*:= u_{k-1} + e_{0,k-1}$ for $k \geq 1$ at hand, define the residual
\begin{equation}\label{def-residual}
	R_T(u_k) = \begin{cases}
		f|_T - \hmcL u_0 |_T, & k = 0, \\
		f|_T - \mcL u_k|_T - \mcN u_{k-1}^* |_T, & k \geq 1, 
	\end{cases}
\end{equation}
for any element $T \in \mcT_k$, and define the jump 
\begin{equation}\label{jump}
	J_E(u_k) := [\valpha \nabla u_k \cdot \vect{n}]|_E := \valpha\nabla u_k|_{T_1} \cdot \vect{n}_1 + \valpha \nabla u_k |_{T_2} \cdot \vect{n}_2,
\end{equation}
for any side $E \in \mcE_k^I$, where $\mcE_k^I$ denotes the set of all interior sides being an edge ($d = 2$) or a face ($d = 3$) in $\mcT_k$, and  
the elements $T_1,T_2 \in \mcT_k$ share $E$ with the unit outward normal vectors $\vect{n}_1,\vect{n}_2$. 

We define the following local a posteriori error estimator: 
\begin{equation}\label{eta-T}
	\eta_k(u_k,T) := \Big(h_T^2 \|R_T(u_k)\|_{0,T}^2 + h_T 
	\sum_{E \subset \partial T \cap \Omega}\LtwE{J_E(u_k)}^2\Big)^\frac{1}{2}, 
\end{equation}	
 and for any element patch $S \subset \mcT_k$, define 
\begin{equation}\label{eta-D}
	\eta_k(u_k,S) := \Big(\sum_{T \in S} \eta_k^2 (u_k,T)\Big)^\frac{1}{2}.   
\end{equation}

\subsection{Procedure \textbf{MARK}}
This procedure aims to select which elements will be refined in the current mesh in term of the local contribution of the error estimator. We consider the D\"orfler marking strategy \cite{Dorfler96:1106}. Given a marking parameter $\theta \in (0,1)$, Procedure \textbf{MARK} outputs a set of marked elements $\mcM_k \subset \mcT_k$ with minimal cardinality such that 
\begin{equation}\label{mark}
	\eta_k^2(u_k,\mcM_k) \geq \theta \eta_k^2(u_k,\mcT_k).
\end{equation}

\subsection{Procedure \textbf{REFINE}}\label{part-refine}
Starting from a conforming and shape-regular initial mesh $\mcT_0$, a conforming mesh $\mcT_{k+1}$ is obtained from $\mcT_k$ by the local mesh refinement. To generate $\mcT_{k+1}$. we first bisect all marked elements, i.e., divide each element in $\mcM_k \subset \mcT_k$ into two equal volume parts, to get a new mesh $\mcT_{k+1}'$. In general, hanging nodes exist in $\mcT_{k+1}'$, so we have to refine some elements in $\mcT_{k+1}'$ to generate a conforming mesh $\mcT_{k+1}$. 

We assume that the mesh family $\{\mcT_k\}_{k \geq 0}$ generated by the above procedure is uniformly shape-regular. The possible choice of the local mesh refinement is the newest vertex bisection method \cite{Stevenson08:227}. Furthermore, note that meshes $\mcT_{k+1}$ and $\mcT_k$ is nested in the sense that the finite element spaces $\mbV_k \subset \mbV_{k+1}$. 

\subsection{Adaptive algorithm}

Based on the above descriptions of each procedure of adaptivity loop \eqref{loop}, we formulate our CATGFEM as follows: 
\begin{breakablealgorithm}
	\renewcommand{\algorithmicrequire}{\textbf{Input:}}
	\renewcommand{\algorithmicensure}{\textbf{Output:}}
	\caption{Correction adaptive two-grid finite element method (CATGFEM) }\label{alg-AFEMC}
	\begin{algorithmic}
		\REQUIRE The initial mesh $\mcT_0$, the source term $f$, the stopping parameter $tol > 0$, the marking parameter $\theta \in (0,1)$, and the adaptivity counter $k = -1$.
		\ENSURE The adaptive mesh $\mcT_K$ and the finite element solution $u_K$ as well as the corresponding error estimator $\eta_K(u_K,\mcT_K)$.
		\REPEAT
		\STATE $k = k + 1$. 
		\STATE \textbf{SOLVE} the finite element solution $u_k$ (note that $u_0 = \hu_0$) on the mesh $\mcT_k$ by the discretization \eqref{solve-SAFEM} if $k = 0$, or the discretization \eqref{solve-CATGFEM-1} if $k = 1$, otherwise, discretizations \eqref{mysolver-k-1}--\eqref{mysolver-k-2}.
		\STATE \textbf{ESTIMATE} the energy-norm error $\anorm{u-u_k}$ by the error estimator $\eta_k(u_k,\mcT_k)$ with the elementwise contribution $\eta_k(u_k,T)$ computed by the representations \eqref{eta-T}--\eqref{eta-D}. 
		\STATE\textbf{MARK} elements in the subset $\mcM_k \subset \mcT_k$ with minimal cardinality satisfying \eqref{mark}.
		\STATE\textbf{REFINE} all elements in the set $\mcM_k$ using the bisection method to produce the new mesh $\mcT_{k+1}$.
		\UNTIL{$\eta_k(u_k,\mcT_k) \leq tol$}
		\RETURN $\mcT_K = \mcT_k$, $u_K = u_k$, $\eta_K(u_K,\mcT_K) = \eta_k(u_k,\mcT_k)$. 
	\end{algorithmic}  
\end{breakablealgorithm}

\begin{remark}It is worth to note that the ATGFEM and our CATGFEM both adopt the D\"orfler marking strategy and the bisection method in Procedures \textbf{MARK} and \textbf{REFINE}, respectively. Their main difference is Procedure \textbf{SOLVE}; see Remark \ref{rem-solve-difference}. Note that, for Procedure \textbf{ESTIMATE}, the construction of the error estimator of the CATGFEM follows the idea of the ATGFEM. 

\end{remark}

\section{Convergence analysis}\label{sect3}

In this section, we first prove a key auxiliary result called the $L^2$-lifting property of the error of the corrected discrete solution, which will be frequently used in the analysis below. We are then devoted to verifying three components of contraction analysis: quasi-orthogonality, reliability and estimator reduction. Finally, we prove a weaker version of contraction of the quasi-error compared with the classical version, and derive the convergence of the proposed CATGFEM (Algorithm \ref{alg-AFEMC}). 

\subsection{$L^2$-lifting property} 
The following lemma indicates that, under the condition that the initial mesh size $h_0$ is sufficiently small, the $L^2$-norm error of the corrected discrete solution $u_{k-1}^* = u_{k-1} + e_{0,k-1}$ is a higher-order quantity of the energy-norm error of the discrete solution $u_{k-1}$. It is precisely this lemma that enables us to establish a series of stronger theoretical results, in particular more rigorous contraction and convergence properties for the CATGFEM, in comparison with the ATGFEM in \cite{LiYkZhangY21:A908}.

\begin{lemma}[$L^2$-lifting property]\label{lem-L2-lifting}
	Let $u$, $e_{0,k-1}$ and $u_k$ be the solutions of \eqref{model-VP}, \eqref{mysolver-k-1} and \eqref{mysolver-k-2}, and keep $u_{k-1}^*=e_{0,k-1} + u_{k-1}$ in mind. Then, there holds 
		\begin{equation}\label{result-L2-lifting}
			\Ltw{u-u_{k-1}^*} \leq \kappa(h_0) \anorm{u-u_{k-1}}
		\end{equation}
with a positive constant $\kappa(h_0) = \kappa(h_0,\Omega,\valpha,\vbeta,\gamma)$ satisfying
		\begin{equation}\label{property-kappaH}
			\lim_{h_0 \rightarrow 0} \kappa(h_0) = 0.  
		\end{equation}
\end{lemma}

\begin{remark}Note that Lemma \ref{lem-L2-lifting} mathematically gives the fact that the $L^2$-norm error $\|u-u_{k}^*\|_0$ is a higher-order quantity of the energy-norm error $\anorm{u-u_{k}}$ provided that the initial mesh size $h_0$ is sufficiently small, compared with \cite{LiYkZhangY21:A908}, where the $L^2$-norm error $\|u-\tu_k\|_0$ may be only empirically seen as a higher-order quantity of the energy-norm errror $\anorm{u-\tu_k}$ when $h_0$ is sufficiently small, here $\tu_k$ is the solution of \eqref{solve-ATGFEM}. 
\end{remark}

For the purpose of the proof of Lemma \ref{lem-L2-lifting}, we need the following projection operator $\hmcP_0: \; H_0^1(\Omega) \rightarrow \mbV_0$ defined by 
\begin{equation}\label{projection-H}
	\ha(v,w_0) = \ha(\hmcP_0 v,w_0), \quad \forall \; v \in H_0^1(\Omega),\; w_0 \in \mbV_0. 
\end{equation} 
The projection operator $\hmcP_0$ is well-defined if the initial mesh size $h_0$ is sufficiently small according to \eqref{est-regularity} and \eqref{infsup-FEM}; see \cite[Lemma 2.3]{XuJC96:1759}. Moreover, $\hmcP_0$ is bounded with respect to the energy-norm $\anorm{\cdot}$: 
	\begin{equation}\label{projection-H-bound}
		\anorm{\hmcP_0 v} \lesssim \anorm{v}.
	\end{equation}
In fact, by \eqref{infsup-FEM}, \eqref{projection-H} and \eqref{ha-boundedness}, we have 
	\begin{align}
		\anorm{\hmcP_0 v} \lesssim~& \sup_{w_0 \in \mbV_0} \dfrac{\ha(\hmcP_0 v,w_0)}{\anorm{w_0}} = \sup_{w_0 \in \mbV_0} \dfrac{\ha(v,w_0)}{\anorm{w_0}} \lesssim \anorm{v}. 
	\end{align}

\begin{proof}[Proof of Lemma \ref{lem-L2-lifting}]
	Substracting \eqref{mysolver-k-1-equivalence} from \eqref{model-VP} yields 
	\begin{equation}\label{result-orth}
		\ha(u-u_{k-1}^*,v_0) = \langle f, v_0 \rangle- \ha(e_{0,k-1}+u_{k-1},v_0) = 0, \quad \forall \; v_0 \in \mbV_0. 
	\end{equation}
By \eqref{result-orth}, \eqref{ha-boundedness} and the norm equivalence \eqref{equivalence-norm}, as the proof of the Aubin-Nitsche lemma \cite[Theorem 3.2.4]{Ciarlet02Book}, we have 
	\begin{align}
		\|u-u_{k-1}^*\|_0 &\lesssim  \|u-u_{k-1}^*\|_a \sup_{g \in L^2(\Omega)} \frac{1}{\|g\|_0}\inf_{\phi_0 \in \mbV_0}\|\phi_g - \phi_0\|_a \label{aubin-nitsche-1}\\
		&= \anorm{u-u_{k-1}^*}\sup_{g \in L^2(\Omega),\|g\|_0 = 1}\inf_{\phi_0 \in \mbV_0}\|\phi_g - \phi_0\|_a, \label{aubin-nitsche-2}
	\end{align}
where for given $g \in L^2(\Omega)$, $\phi_g$ is the unique solution of the following problem: Find $\phi_g \in H_0^1(\Omega)$ such that 
	\begin{equation}\label{aux-VP}
		(\hmcL^* \phi_g,v) = (\hmcL v,\phi_g) = \ha(v,\phi_g) = (g,v), \quad \forall \; v \in H_0^1(\Omega), 
	\end{equation}
and hence $\phi_g = (\hmcL^*)^{-1}g$. 
Moreover, from  \cite[Eqn. (2.4.5)]{Ciarlet02Book}, as the initial mesh size $h_0$ tends to zero, there exists a family of subspaces of $H_0^1(\Omega)$ such that  
	\begin{equation*}
		\lim_{h_0 \rightarrow 0}\sup_{g \in L^2(\Omega),\|g\|_0 = 1}\inf_{\phi_0 \in \mbV_0}\anorm{(\hmcL^*)^{-1}g-\phi_0} = 0. 
	\end{equation*}

For the result \eqref{result-L2-lifting}, it remains to prove that 
\begin{equation*}
	\|u-u_{k-1}^*\|_a \lesssim \anorm{u-u_{k-1}}. 
\end{equation*}
Collecting \eqref{projection-H}, \eqref{model-VP} and \eqref{mysolver-k-1-equivalence} yields 
	\begin{equation*}
		\ha(\hmcP_0 u,v_0) = \ha(u_{k-1}^*,v_0) = \ha(\hmcP_0(u_{k-1}+e_{0,k-1}),v_0), \quad \forall \; v_0 \in \mbV_0, 
	\end{equation*}
which implies that  
	\begin{equation*}
		\hmcP_0 u = \hmcP_0 (u_{k-1} + e_{0,k-1}).
	\end{equation*}
Noting thtat $e_{0,k-1} \in \mbV_0$, we have $\hmcP_0 u = \hmcP_0 u_{k-1} + e_{0,k-1}$, which implies
\begin{equation}\label{result-identity-eH}
	 e_{0,k-1} = \hmcP_0 (u-u_{k-1}).
	\end{equation} 
Combining this identity \eqref{result-identity-eH} with \eqref{infsup-VP}, \eqref{ha-boundedness} and \eqref{projection-H-bound} and keeping $u_{k-1}^* = u_{k-1}+e_{0,k-1}$ in mind lead to that 
	\begin{align}
		\anorm{u-u_{k-1}^*} &\lesssim \sup_{v \in H_0^1(\Omega)}\frac{\ha(u-u_{k-1}^*,v)}{\anorm{v}} \nonumber\\
							&= \sup_{v \in H_0^1(\Omega)}\frac{\ha(u-u_{k-1}-\hmcP_0 (u-u_{k-1}),v)}{\anorm{v}}\nonumber\\
						   &\lesssim \anorm{u-u_{k-1}}, \nonumber
    \end{align}
\end{proof}

\begin{remark}If the solution $\phi_g$ of \eqref{aux-VP} satisfies the extra regularity property $\phi_g \in H^{1+s}(\Omega) \; (0 < s \leq 1)$, then, applying the approximation property of $\hmcP_0$ \cite[Lemma 2.4]{XuJC96:1759} to \eqref{aubin-nitsche-1} and using the regularity estimate yield that 
	\begin{equation*}
		\|u-u_{k-1}^*\|_0 \lesssim h_0^s \anorm{u-u_{k-1}}. 
	\end{equation*} 
However, the extra regularity property is generally not satisfied in the presence of singularities, so we consider the lowest regularity property $\phi_g \in H_0^1(\Omega)$ throughout this paper. 
\end{remark}

\subsection{Quasi-orthogonality}
The following lemma establishes quasi-orthogonality for the energy-norm error without any remainder term, owing to Lemma~\ref{lem-L2-lifting}.
\begin{lemma}[Quasi-orthogonality]\label{lem-quasi-orth}
	Let $u$, $e_{0,k-1}$ and $u_k$ be the solutions of \eqref{model-VP}, \eqref{mysolver-k-1} and \eqref{mysolver-k-2}, and keep $u_{k-1}^*=e_{0,k-1} + u_{k-1}$ in mind. 
	Then, there holds
	\begin{equation}\label{result-quasi-orth}
		\anorm{u-u_k}^2 \leq \dfrac{1}{1-c_\qoo \kappa(h_0)} \anorm{u-u_{k-1}}^2 - \anorm{u_{k-1}-u_k}^2.  
	\end{equation}
with a constant $c_\qoo = c_\qoo(\Omega,\valpha,\vbeta,\gamma)> 0$ and the initial mesh size $h_0$ is sufficiently small such that $c_\qoo \kappa(h_0) < 1$, 
\end{lemma}
\begin{proof}
	We show that  
	\begin{align}
		\anorm{u-u_k}^2 &= a(u-u_{k-1}+u_{k-1}-u_k,u-u_{k-1}+u_{k-1}-u_k)
			   \nonumber\\
			&= \anorm{u-u_{k-1}}^2 + \anorm{u_k-u_{k-1}}^2 + 2a(u-u_{k-1},u_{k-1}-u_k)
               \nonumber\\
            &= \anorm{u-u_{k-1}}^2 - \anorm{u_k-u_{k-1}}^2 + 2a(u-u_k,u_{k-1}-u_k). 
	           \label{step1-quasi-orth}
	\end{align}
For the third term in the right-hand side of \eqref{step1-quasi-orth}, using  \eqref{mysolver-k-1-equivalence}, the following estimate of the lower-order term $\mcN$ is defined in \eqref{bilinear-form-N}
\begin{equation}\label{bilinear-N-B0}
	N(v,w) \lesssim \Ltw{v}\Hon{w}, \quad \forall \; v,w \in H_0^1(\Omega).
\end{equation}
Then the result \eqref{result-L2-lifting} from Lemma \ref{lem-L2-lifting} and the norm equivalence \eqref{equivalence-norm} leads to 
	\begin{align}
		2a(u-u_k,u_{k-1}-u_k) &= -2N(u-u_{k-1}^*,u_{k-1}-u_k)
				\nonumber\\
		  &\lesssim \Ltw{u-u_{k-1}^*}\Hon{u_{k-1}-u_k}
		  	    \nonumber\\
	      &\lesssim \kappa(h_0) \anorm{u-u_{k-1}}\anorm{u_{k-1}-u_k}
	      		\nonumber\\
	      &\leq 2\wtC_\qoo\kappa(h_0) \Big(\anorm{u-u_k}\anorm{u_{k-1}-u_k} 
	      						  + \anorm{u_{k-1}-u_k}^2\Big)
	            \label{step2-quasi-orth}
	\end{align}
with a positive constant $\wtC_\qoo = \wtC_\qoo(\Omega,\valpha,\vbeta,\gamma)$. 
Applying the Young inequality with any $\delta > 0$ to the first term in the right-hand side of \eqref{step2-quasi-orth} and substituting the arising result into \eqref{step1-quasi-orth} yield that 
	\begin{equation*}
		\Big(1-\delta \wtC_\qoo \kappa(h_0)\Big)\anorm{u-u_k}^2 \leq \anorm{u-u_{k-1}}^2  - \Big(1-\delta^{-1}\wtC_\qoo \kappa(h_0) - 2\wtC_\qoo\kappa(h_0)\Big)\anorm{u_{k-1}-u_k}^2. 
	\end{equation*}
Setting
	\begin{equation*}
		1 - \delta \wtC_\qoo \kappa(h_0) = 1-\delta^{-1}\wtC_\qoo \kappa(h_0) - 2\wtC_\qoo\kappa(h_0) \quad \text{implies} \quad \delta = \sqrt{2} + 1. 
	\end{equation*}
Let the initial mesh size $h_0$ be sufficiently small such that $(\sqrt{2}+1)\wtC_\qoo \kappa(h_0) < 1$. We conclude the proof. 
\end{proof}

\subsection{Reliability}
 We first give the following approximation properties of the quasi-interpolation operator; see, e.g.,  \cite[Lemma 1.3]{verfurth13Book}), which is a necessary tool to prove the reliability of the error estimator. A typical example of the quasi-interpolation operator is the Scott-Zhang interpolant \cite{ScottZhang90:483}. 
\begin{lemma}\label{Lem:4.1} 
	There exists a quasi-interpolation operator $\Pi_{k}:\,H_0^1(\Omega )\rightarrow \mbV_k$ such that for any $v \in H_0^1(\Omega)$,  any $T\in \mcT_k$, and any $ E \in \mathcal{E}_{k}^I$, there hold 
	\begin{eqnarray}
		\left\|v-\Pi_{k}v\right\|_{0, T} \lesssim h_{T}\left\|\nabla v\right\|_{0, \hat{\omega}_{T}}, \label{operator-approximate-T}\\ 
		\left\|v-\Pi_{k}v\right\|_{0, E} \lesssim h_{T}^{\frac{1}{2}}\left\|\nabla v\right\|_{0, \hat{\omega}_{E}},\label{operator-approximate-E}
	\end{eqnarray} 
where the hidden constants only depend on the uniform shape regularity of the mesh family $\{\mcT_k\}_{k \geq 0}$, and the sets $\hat{\omega}_{T}$ and $\hat{\omega}_{E}$ denote the unions of elements sharing at least a vertex with $T$ and $E$, respectively. 
\end{lemma} 

The following lemma indicates that the global error estimator is the a posteriori upper bound of the energy-norm error, up to some theoretical higher-order quantity. 
	To save notations, we write $\eta_k(u_k)$ instead of $\eta_k(u_k,\mcT_k)$ below. 

\begin{lemma}[Reliability]\label{lem-rel} 
	Let $u$ and $u_k$ be the solutions of \eqref{model-VP} and \eqref{mysolver-k-2}. Then, for $k \geq 1$, there holds 
	\begin{equation}\label{result-rel}
		\anorm{u-u_k} \leq C_\rel \eta_k(u_k) + \widetilde{C}_\rel\kappa(h_0) \anorm{u-u_{k-1}}
	\end{equation}
with positive constants $C_\rel = C_\rel(\Omega,\valpha)$ dependent of the uniform shape regularity of the mesh family $\{\mcT_k\}_{k \geq 0}$ and $\widetilde{C}_\rel = \wtC_\rel(\Omega,\valpha,\vbeta,\gamma)$. 

Specially for $k = 0$, the second term in the right-hand side of \eqref{result-rel} will vanish. 
\end{lemma}

\begin{remark}In \cite[Theorem 3.1]{LiYkZhangY21:A908}, the reliability of the error estimator reads in our setting as follows: 
	\begin{equation}\label{result-rel-ATGFEM}
		\anorm{u-\tu_k} \lesssim \tilde{\eta}_k(\tu_k) + \|u-\tu_{k-1}\|_0, 
	\end{equation}
where $\tilde{\eta}_k(\tu_k)$ is defined in \cite[(3.9)]{LiYkZhangY21:A908}. Notice that the reference \cite{LiYkZhangY21:A908} thinks that \eqref{result-rel-ATGFEM} holds up to an empirical higher-order quantity $\|u-\tu_{k-1}\|_0$ compared with the corresponding energy-norm error $\anorm{u-\tu_{k-1}}$. Therefore, our reliablity \eqref{result-rel} given by Lemme \ref{lem-rel} is a better result than \eqref{result-rel-ATGFEM} due to the theoretical higher-order remainder term. 
	
\end{remark}

\begin{proof}Note that if $k = 0$, we readily obtain $\anorm{u-u_0} \leq C_\rel \eta_0(u_0)$ according to \cite[Eqn. (2.7)]{FeischlFuhrer14:601} and the norm equivalence \eqref{equivalence-norm}. It remains to prove \eqref{result-rel} for $k \geq 1$. 
	
	For any $v \in H_0^1(\Omega)$ and $v_k \in \mbV_k$, using \eqref{model-VP} and \eqref{mysolver-k-2}, we show that 
\begin{align}
	a(u-u_k,v) =~& a(u-u_k,v-v_k) + a(u-u_k,v_k) \nonumber\\ 
		=~& \langle f,v-v_k \rangle - N(u,v-v_k) - a(u_k,v-v_k) - N(u-u_{k-1}^*,v_k)\nonumber \\
		=~& \langle f,v-v_k \rangle - a(u_k,v-v_k) - N(u_{k-1}^*,v-v_k) - N(u-u_{k-1}^*,v).\label{step1-rel}
	\end{align}
Applying integration by parts, the H\"older inequality, the estimate \eqref{bilinear-N-B0}, and the result \eqref{result-L2-lifting} of Lemma \ref{lem-L2-lifting} to \eqref{step1-rel} yields  
\begin{align}
	&a(u-u_k,v) \nonumber\\
	=~& \sum_{T \in  \mcT_k}(f-\mcL u_k - \mcN u_{k-1}^*,v-v_k)_T 
					- \sum_{E \in \mcE_k^I} \int_E[\valpha\nabla u_k\cdot \vect{n}] (v-v_k)\mrd s \nonumber- N(u-u_{k-1}^*,v)\nonumber\\
		   \leq~&\sum_{T \in  \mathcal{T}_k}\|R_T(u_k)\|_{0,T}\|v-v_k\|_{0,T}
				+\sum_{E \in \mathcal{E}_k^I} \|J_E(u_k) \|_{0, E}\| v-v_k\|_{0, E}+C(\vbeta,\gamma)\kappa(h_0) \anorm{u-u_{k-1}}\|v\|_1. \label{step2-rel}
\end{align}
Taking $v_{k}=\Pi_{k}v$, where $\Pi_{k}$ is the quasi-interpolation operator in Lemma \ref{Lem:4.1}, and using the corresponding approximation properties \eqref{operator-approximate-T}--\eqref{operator-approximate-E}, the Cauchy-Schwarz inequality, the uniform shape regularity of the mesh family $\{\mcT_k\}_{k \geq 0}$ and the norm equivalence \eqref{equivalence-norm} lead to that 
\begin{align}
	a(u-u_k,v) \leq~& \sum_{T \in  \mcT_k}\|R_T(u_k)\|_{0,T}\|v-\Pi_k v\|_{0, T}
				      +\sum_{E \in \mcE_k^I} \|J_E(u_k)\|_{0,E}\|v-\Pi_k v\|_{0,E}\nonumber+C(\vbeta,\gamma)\kappa(h_0) \anorm{u-u_{k-1}}\|v\|_1\nonumber\\
		\lesssim~& \Big(\sum_{T \in  \mcT_k}h_T\|R_T(u_k)\|_{0,T}\|\nabla v\|_{0,\homega_T}
				 +\sum_{E \in \mcE_k^I} h_T^\frac12\|J_E(u_k)\|_{0,E}\|\nabla v\|_{0, \homega_E}\Big)\nonumber+\kappa(h_0)\anorm{u-u_{k-1}}\|v\|_1\nonumber\\
		\lesssim~&\Big(\sum_{T \in \mcT_k}h_T^2\|R_T(u_k)\|_{0,T}^2 
				   + \sum_{E \in \mcE_k^I} h_T\|J_E(u_k)\|_{0,  E}^2\Big)^\frac{1}{2}\anorm{v}+\kappa(h_0)\anorm{u-u_{k-1}}\anorm{v},\label{step3-rel}
\end{align}
where the uniform shape regularity implies that 
	\begin{equation*}
		\Big(\sum_{T \in \mcT_k}\LtwD{\nabla v}{\hat{\omega}_T}^2
	+	\sum_{E \in \mcE_k^I}\LtwD{\nabla v}{\homega_E}^2\Big)^\frac{1}{2} \lesssim \Hon{v}. 
	\end{equation*}
Taking $v=u-u_{k}$ in \eqref{step3-rel} yields \eqref{result-rel}. 
\end{proof}

%

\subsection{Estimator reduction}
The following lemma indicates that the error estimator is contractive up to some perturbations. 
\begin{lemma}[Estimator reduction]\label{lem-est-reduce} 
	Let $u$ and $u_k$ be the solutions of \eqref{model-VP} and \eqref{mysolver-k-2}. Then, for $k \geq 2$, there holds
	\begin{align}
		\eta_k^2(u_k) \leqslant~& \rho_\est \eta_{k-1}^2(u_{k-1}) 
+ C_\est\anorm{u_k-u_{k-1}}^2+ C_\rem \kappa(h_0)^2\left(\anorm{u-u_{k-1}}^2 + \anorm{u-u_{k-2}}^2\right)
\label{result-est-reduce}
	\end{align}
with $\rho_\est = \rho_\est(d,\theta) \in (0,1)$, and the positive constants $C_\est = C_\est(d,\Omega,\valpha)$ and $C_\rem =  C_\rem(\vbeta,\gamma)$ both also dependent of the uniform shape regularity of the mesh family $\{\mcT_k\}_{k \geq 0}$. 

Particularly for $k = 1$, the term $C_\rem \kappa(h_0)^2(\anorm{u-u_{k-1}}^2 + \anorm{u-u_{k-2}}^2)$ of \eqref{result-est-reduce} will vanish. 
\end{lemma}

\begin{proof}For the elementwise error estimator \eqref{eta-T}, the following inequality 
	$$
	(b_1^{2}+c_1^{2})^{\frac{1}{2}}-(b_2^{2}+c_2^{2})^{\frac{1}{2}}\leqslant (b_1-b_2) + (c_1 - c_2), \quad \forall \; b_1,b_2,c_1,c_2 \geqslant 0, 
	$$
implies that 
	\begin{align}
	\eta_k(u_k,T) - \eta_k(u_{k-1},T)
		\leq~&h_T\LtwT{R_T(u_k)-R_T(u_{k-1})}+ \sum_{E \subset \partial T\cap\Omega}h_T^\frac{1}{2}\LtwE{J_E(u_k)-J_E(u_{k-1})}. \label{step1-stability}
	\end{align}
By the definition \eqref{def-residual} of $R_T(\cdot)$ and the inverse estimate, the first term in the right-hand side of \eqref{step1-stability} is bounded by 
	\begin{align}
	h_T \LtwT{R_T(u_k)-R_T(u_{k-1})}  \leq~&
		h_T\Big(\LtwT{(\nabla\cdot\valpha)\cdot\nabla(u_k-u_{k-1})}
		+\LtwT{\valpha:\nabla^2(u_k-u_{k-1})}
		\nonumber\\
		&+\LtwT{\vbeta\cdot\nabla(u_{k-1}^*-u_{k-2}^*)}
		 +\LtwT{\gamma(u_{k-1}^*-u_{k-2}^*)}\Big)
		\nonumber\\
		\lesssim~& \LtwT{\nabla(u_k-u_{k-1})} + \LtwT{u_{k-1}^* - u_{k-2}^*}
		\label{step2-stability}
	\end{align}
for $k \geq 2$, where $\nabla^2(\cdot)$ denotes the Hessian matrix. Note that, for $k = 1$, the last term in the right-hand side of \eqref{step2-stability} will vanish due to $u_0^* = u_0 + e_{0,0} = u_0$ from Remark \ref{rem-CATG-1}. 

For the estimation of the second term in the right-hand side of \eqref{step1-stability}, we note that for any side $E \in \partial T \cap \Omega$, there exists an element $T' \in \mcT_k$ such that $T' \cap T=E$. By the definition \eqref{jump} of the jump, the trace theorem and the inverse estimate, we show that 
\begin{align}
	h_T^{\frac{1}{2}}\LtwE{J_E(u_k)-J_E(u_{k-1})}\leq~&
		h_T^\frac12(\LtwE{\valpha\nabla(u_k-u_{k-1})|_{\partial T \cap E}\cdot\vn_T}+\LtwE{\valpha\nabla(u_k-u_{k-1})|_{\partial T' \cap E}\cdot\vn_{T'}})
			\nonumber\\
		\lesssim~& \LtwD{\nabla(u_k-u_{k-1})}{T \cup T'}. 
			\nonumber
\end{align}
Then, it holds that 
	\begin{align}
		\sum_{E \in \partial T \cap \Omega}h_T^\frac12\LtwE{J_E(u_k)-J_E(u_{k-1})}
		\lesssim~&(d+1)\LtwD{\nabla(u_k-u_{k-1})}{\omega_T}, 
		\label{step3-stability}
	\end{align}
where $\omega_T$ denotes the set of elements sharing the edge $E \in \partial T \cap \Omega$ with the element $T$.  
Substituting \eqref{step2-stability} and \eqref{step3-stability} into \eqref{step1-stability} yields that 
	\begin{align}
		\eta_k(u_k,T) \leq~&  \eta_k(u_{k-1},T) +  C_1\LtwD{\nabla (u_k-u_{k-1})}{\omega_T} + C_2\LtwT{u_{k-1}^*-u_{k-2}^*} \nonumber
	\end{align}
for $k \geq 2$ with positive constants $C_1= C_1(d,\Omega,\valpha)$ and $C_2 = C_2(\vbeta,\gamma)$ both also dependent of the uniform shape regularity of the mesh family $\{\mcT_k\}_{k \geq 0}$, where the last term in the above right-hand side will vanish for $k = 1$. 

Applying the Young inequality with $\veps > 0$ to the square of the above inequality, summing the arising result over all $T \in \mcT_k$ and using the norm equivalence \eqref{equivalence-norm} as well as \eqref{result-L2-lifting} of Lemma \ref{lem-L2-lifting} lead to that 
\begin{align}
	\eta_k^2( u_k) \leq~& (1+\veps)\eta_k^2(u_{k-1})
		+2(1+\veps^{-1})\Big((d+2)C_1^2\Ltw{\nabla(u_k-u_{k-1})}^2 
		+C_2^2\Ltw{u_{k-1}^*-u_{k-2}^*}^2\Big)\nonumber\\
		\leq~&(1+\veps)\eta_k^2(u_{k-1}) + 2(1+\veps^{-1})(d+2)C_1^2C(\Omega,\valpha)
		\anorm{u_k-u_{k-1}}^2 \nonumber\\
		&+4(1+\veps^{-1})C_2^2\kappa(h_0)^2(\anorm{u-u_{k-1}}^2 + \anorm{u-u_{k-2}}^2)
		\label{step4-stability} 
\end{align}
for $k \geq 2$, where we use the fact that the sum of norms on $\omega_T$ covers each element at least $d+2$ times. While the term $4(1+\veps^{-1})C_2^2\kappa(h_0)^2(\anorm{u-u_{k-1}}^2 + \anorm{u-u_{k-2}}^2)$ in \eqref{step4-stability} will vanish for $k=1$. 

 Furthermore, it is easy to obtain that 
 \begin{equation}\label{reduction}
 	\eta_k^2(u_{k-1}) \leq \rho_\reduce\eta_{k-1}^2(u_{k-1}), 
 \end{equation}
 with the constant $\rho_\reduce:= 1 - 2^{-1/d}\theta$. Subsituting \eqref{reduction} into \eqref{step4-stability} and taking $\veps$ sufficiently small such that 
 $\rho_\est := (1+\veps)\rho_\reduce \in (0,1)$ conclude the proof. 
%
\end{proof}

\subsection{Contraction and convergence}
Now, by Lemma \ref{lem-quasi-orth}, Lemma \ref{lem-rel} and Lemma \ref{lem-est-reduce}, we are in a position to prove the following contraction result of the quasi-error, which shows that some sum of quasi-errors on the successively two levels is contractive. 
\begin{theorem}[Contraction]\label{thm-contract}
	Let $u$ be the solution of the weak formulation \eqref{model-VP}, and $\{u_k,\eta_k(u_k)\}_{k \geq 0}$ be the sequence of the discrete solution
	 and the error estimator obtained by Algorithm \ref{alg-AFEMC}. Then, there exist constants $\rho_\ctr = \rho_\ctr(\rho_\est,C_\est,C_\rem,C_\rel,\wtC_\rel,c_\qoo,\kappa(h_0)) 
	 \in (0,1)$ and $\tau > 0$ having the same dependence as $\rho_\ctr$ such that 
	\begin{equation}\label{result-contract}
	E_{k,\sigma}^2 + \tau \kappa(h_0)^2 E_{k-1,\sigma}^2 \leq \rho_\ctr^2 
	(E_{k-1,\sigma}^2 + \tau\kappa(h_0)^2 E_{k-2,\sigma}^2), 
	\end{equation}
where $E_{k,\sigma}^2 = \anorm{u-u_k}^2 + \sigma \eta_k^2(u_k)$ with the constant $\sigma = C_\est^{-1}$. In particular, there holds 
	\begin{equation}\label{result-contract-1-0}
		E_{1,\sigma}^2 \leq \rho_0^2 E_{0,\sigma}^2, 
	\end{equation}
with the constant $\rho_0 = \rho_0(\rho_\est,C_\est,C_\rel,c_\qoo,\kappa(h_0))$.
\end{theorem}

\begin{proof}For convenience, we write $\eta_k(u_k)$ to $\eta_k$. Adding the quasi-orthogonality \eqref{result-quasi-orth} and the $\sigma = C_\est^{-1}$  multiple of the estimator reduction \eqref{result-est-reduce} leads to that 
	\begin{align}
		\anorm{u-u_k}^2 + \sigma\eta_k^2
		\leq~& \Big(\dfrac{1}{1-c_\qoo \kappa(h_0)} + \sigma C_\rem\kappa(h_0)^2\Big)\anorm{u-u_{k-1}}^2 + \sigma \rho_\est\eta_{k-1}^2+ \sigma C_\rem\kappa(h_0)^2\anorm{u-u_{k-2}}^2. \label{step1-contract}
	\end{align}
Introducing a free parameter $\nu > 0$ and invoking the reliability \eqref{result-rel} yield that 
	\begin{align}
		\anorm{u-u_k}^2 + \sigma\eta_k^2
		\leq~& \Big(\dfrac{1}{1-c_\qoo \kappa(h_0)} + \dfrac{C_\rem \kappa(h_0)^2 - \nu}{C_\est}\Big)\anorm{u-u_{k-1}}^2 + \sigma (\rho_\est + 2C_\rel^2 \nu) \eta_{k-1}^2 \nonumber\\
		&+ \dfrac{C_\rem+2\nu \widetilde{C}_\rel^2}{C_\est}\kappa(h_0)^2\anorm{u-u_{k-2}}^2. \label{step2-contract}
	\end{align}
Noting that $\rho_\est \in (0,1)$, we can choose $\nu > 0$ sufficiently small such that 
	\begin{equation*}
		\rho_\est + 2C_\rel^2 \nu \in (0,1). 
	\end{equation*}
 Also, we can make $\kappa(h_0)$ sufficiently small (i.e., the initial mesh size $h_0$ is sufficiently small) such that $C_\rem \kappa(h_0)^2 < \nu$ and 
	\begin{equation*}
		\dfrac{1}{1-c_\qoo \kappa(h_0)} + \dfrac{C_\rem \kappa(h_0)^2 - \nu}{C_\est} \in (0,1). 
	\end{equation*}
Therefore, it holds 
	\begin{equation}\label{step3-contract}
	E_{k,\sigma}^2 \leq \rho^2 E_{k-1,\sigma}^2 + C_\ctr \kappa(h_0)^2 \anorm{u-u_{k-2}}^2,
	\end{equation}
where $E_{k,\sigma}^2:=\anorm{u-u_k}^2 + \sigma \eta_k^2$, and the constants 
	\begin{align}
		\rho^2 &:= \max\{\dfrac{1}{1-c_\qoo \kappa(h_0)} + \dfrac{C_\rem \kappa(h_0)^2 - \nu}{C_\est},\rho_\est + 2C_\rel^2 \nu\} \in (0,1), \nonumber\\
		C_\ctr &:= \dfrac{C_\rem+2\nu \widetilde{C}_\rel^2}{C_\est}. \nonumber
	\end{align}
In particular, for the above deduction for $k = 1$, the last term in the right-hand side of \eqref{step3-contract} will vanish due to Lemma \ref{lem-est-reduce} for $k=1$ and the constant $\rho^2$ will become 
	\begin{equation*}
		\rho_0^2 := \max\left\{\dfrac{1}{1-c_\qoo \kappa(h_0)} - C_\est^{-1}\nu,~\rho_\est + 2C_\rel^2 \nu\right\} \in (0,1).
	\end{equation*}	
In order to obtain the following contraction 
	\begin{equation*}
		E_{k,\sigma}^2 + \tau \kappa(h_0)^2 E_{k-1,\sigma}^2 \leq \rho_\ctr^2 \left(E_{k-1,\sigma}^2 + \tau\kappa(h_0)^2 E_{k-2,\sigma}^2\right), 
	\end{equation*}
where the constants $\tau > 0$ and $\rho_\ctr \in (0,1)$ will be determined, we know that 
	\begin{equation*}
		\rho^2 = \rho_\ctr^2 - \tau \kappa(h_0)^2, \quad\text{and}\quad C_\ctr = \rho_\ctr^2 \tau, 
	\end{equation*}
which implies that 
	\begin{equation*}
		\rho_\ctr^2 = \dfrac{\rho^2 + \sqrt{\rho^4 + 4C_\ctr \kappa(h_0)^2}}{2}, 
		\quad \text{and} \quad \tau = \dfrac{2C_\ctr}{\rho^2 + \sqrt{\rho^4 + 4C_\ctr \kappa(h_0)^2}}. 
	\end{equation*}
The sufficiently small $\kappa(h_0)$ can guarantee $\rho_\ctr \in (0,1)$. 
\end{proof}

\begin{remark} Although the contraction \eqref{result-contract} of the CATGFEM is a weaker result compared with the classical contraction of the SAFEM that is the quasi-error is constractive with respect to the adaptivity iteration counter $k$, it is a necessary relaxation since note that the CATGFEM is based on the efficient inexact numerical solver while the SAFEM is based on the inefficient exact solver. 

If the initial mesh size $h_0$ is sufficiently small such that 
	\begin{equation}
		\kappa(h_0)\anorm{u-u_{k-2}}^2 \leq \anorm{u-u_{k-1}}^2, \quad \forall \; k \geq 2, 
	\end{equation}
holds, then \eqref{step3-contract} will become the classical contraction 
	\begin{equation}\label{result-convergence}
		E_{k,\sigma}^2 \leq \rho_1^2 E_{k-1,\sigma}^2, 
	\end{equation}
with the constant 
	\begin{equation*}
		\rho_1^2 := \max\left\{\dfrac{1}{1-c_\qoo \kappa(h_0)} + \dfrac{C_\rem \kappa(h_0)^2 - \nu}{C_\est} + C_\ctr \kappa(h_0),~\rho_\est + 2C_\rel^2 \nu\right\} \in (0,1).
	\end{equation*}
\end{remark}
Theorem \ref{thm-contract} immediately derives the convergence of our proposed CATGFEM (Algorithm \ref{alg-AFEMC}).
	\begin{corollary}[Convergence]\label{cor-convergence}
		Under the conditions of Theorem \ref{thm-contract}, then there holds 
		\begin{equation*}
			\anorm{u-u_k}^2 + \sigma\eta_k^2(u_k) \leq \rho_\conv^{2k}E_{0,\sigma}^2\rightarrow 0 \;(k \rightarrow \infty), 
		\end{equation*}
	with the constant $\rho_\conv^{2k} = \rho_\ctr^{2k-2}\left(\rho_0^2 + \tau \kappa(h_0)^2\right)$, where the sufficiently small initial mesh size $h_0$ can make the constant $\kappa(h_0)$ sufficiently small such that 
	\begin{equation*}
		\rho_\conv = \rho_\ctr\sqrt[2k]{\frac{\rho_0^2+\tau \kappa(h_0)^2}{\rho_\ctr^2}} \in (0,1). 
	\end{equation*} 
	\end{corollary}

	\begin{remark}In \cite[Theorem 3.3]{LiYkZhangY21:A908}, the convergence of the ATGFEM reads in our settings: there exist constants $\tilde{\sigma} > 0$ and $0 < \tilde{\rho} < 1$ such that 
		\begin{align}
			\anorm{u-\tu_k}^2 + \tilde{\sigma}\tilde{\eta}_k^2(\hu_k) 
			\leq ~& \tilde{\rho}^{2k}\left(\anorm{u-\tu_0}^2+\tilde{\sigma}\tilde{\eta}_k^2(\tu_k)\right) + C\sum_{i=1}^{k-1}\|u-\tu_{k-1-i}\|_0^2 \tilde{\rho}^{2i}
			 + C\sum_{i=1}^{k-1}\|u-\tu_{k-2-i}\|_0^2 \tilde{\rho}^{2i}, \label{result-convergence-ATGFEM}
		\end{align}
where the two sums in the right-hand side of \eqref{result-convergence-ATGFEM} are only empirically regarded as higher-order quantities, which is not theoretically justified. Therefore, our convergence result \eqref{result-convergence} is better and more rigorous than \eqref{result-convergence-ATGFEM} due to no remainder terms. 
This is essentially attributed to the key $L^2$-lifting property from Lemma \ref{lem-L2-lifting}, which allows us to rigorously prove the contraction result \eqref{result-contract} also without remainder terms. 
	\end{remark}

\section{Numerical experiments}\label{sect-test}
In this section, we show the effectiveness and robustness of our proposed CATGFEM (Algorithm \ref{alg-AFEMC}) with several numerical examples. In all numerical examples the linear FEM is considered.  

In the first example,  we compare the convergence behaviors of the SAFEM, the ATGFEM and the CATGFEM with the fixed marking parameter, and of the ATGFEM as well as the CATGFEM with different marking parameters. 

\begin{example}\label{test_peak-fracalpha}	
	We consider the unit square domain $\Omega = (0, 1)\times (0, 1)$ in model problem \eqref{model-1}, and select the coefficients 
		\begin{equation*}
			\valpha =  \left[\begin{array}{cc}
				1 & 0\\
				0 & 1
			\end{array}\right]=\mbI,
			\quad 
			\vbeta = \left[\begin{array}{c}
				1 \\
				1
			\end{array}\right], 
			\quad
			\gamma = -20,   
		\end{equation*}
		and the analytical solution 
		\begin{equation*}
			u = \Big(x^2 + y^2\Big)^{1/5} + \Big((x-1)^2+(y-1)^2\Big)^{1/5}, 
		\end{equation*}
all of which determine the source term $f$. Furthermore, the uniform initial mesh size $h_0$ is chosen to be $1/10$. 
\end{example}

 Figure \ref{fig_test_peak-fracalpha_algo-STD-ATG-CATG_err_mark-0.3_rct--20} displays the convergence history on the energy-norm errors of the SAFEM, the ATGFEM and the CATGFEM with the fixed marking parameter $\theta = 0.3$. Figure \ref{fig_test_peak-fracalpha_rct--20_mark_algo-ATG-CATG_err}  depicts the corresponding results of the ATGFEM and the CATGFEM with different marking parameters. 
 
 From Figure \ref{fig_test_peak-fracalpha_algo-STD-ATG-CATG_err_mark-0.3_rct--20}, we see that the error of the CATGFEM exhibits the optimal convergence, and is almost equal to that of the SAFEM at the same number of elements. However, the error of the ATGFEM nearly remain unchanged from some number of elements. 
 The above observations reflect that the CATGFEM performs well like the SAFEM but the ATGFEM does not work in this example.

\begin{figure}[!ht]
	\centering
		\includegraphics[scale=0.45]{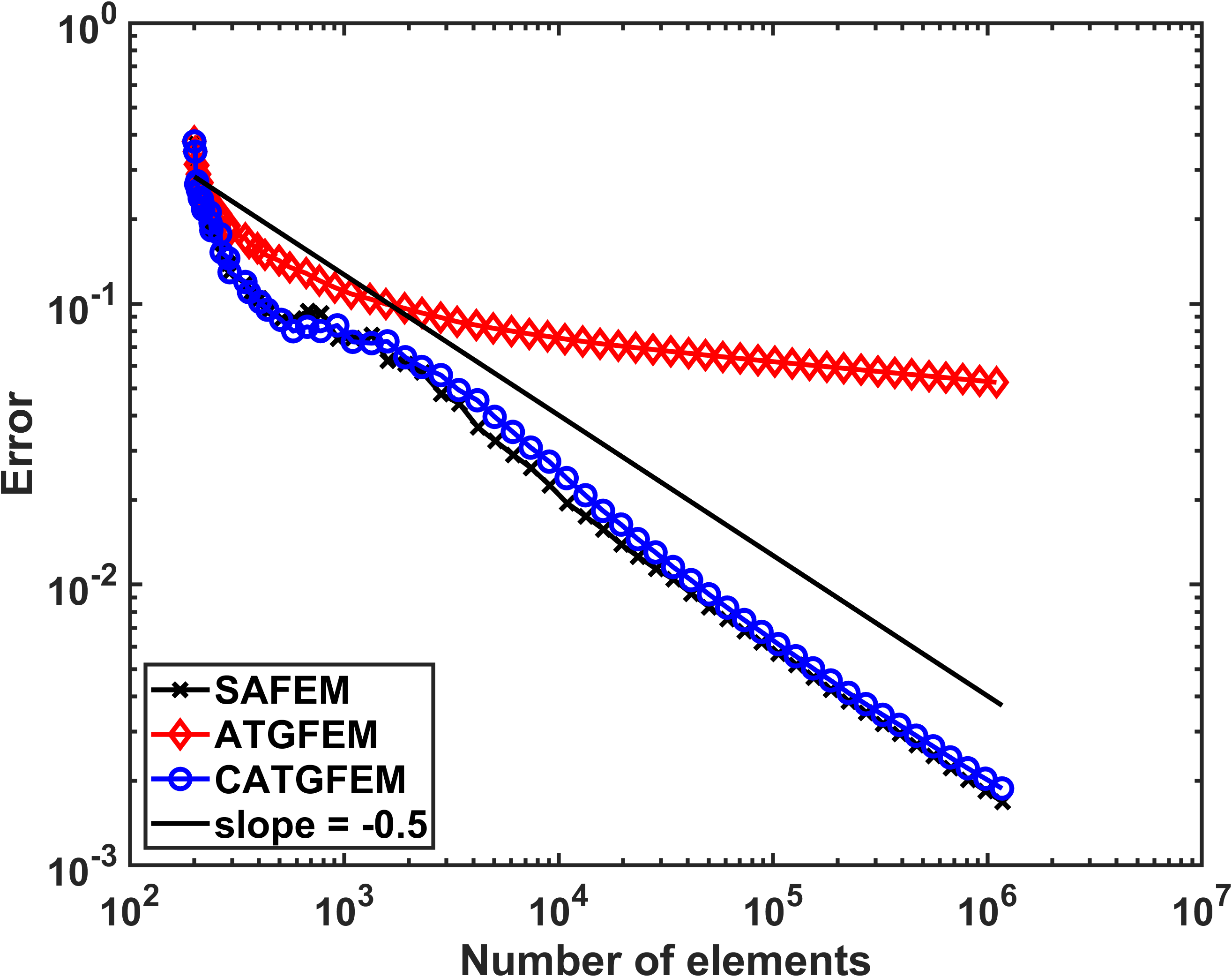}
	\caption{Convergence history on the energy-norm errors of SAFEM, ATGFEM and CATGFEM with marking parameter $\theta = 0.3$ in Example \ref{test_peak-fracalpha}.}\label{fig_test_peak-fracalpha_algo-STD-ATG-CATG_err_mark-0.3_rct--20}
\end{figure}

\begin{figure}[!ht]
	\centering
	\subfigure[ATGFEM]{
		\includegraphics[scale=0.32]{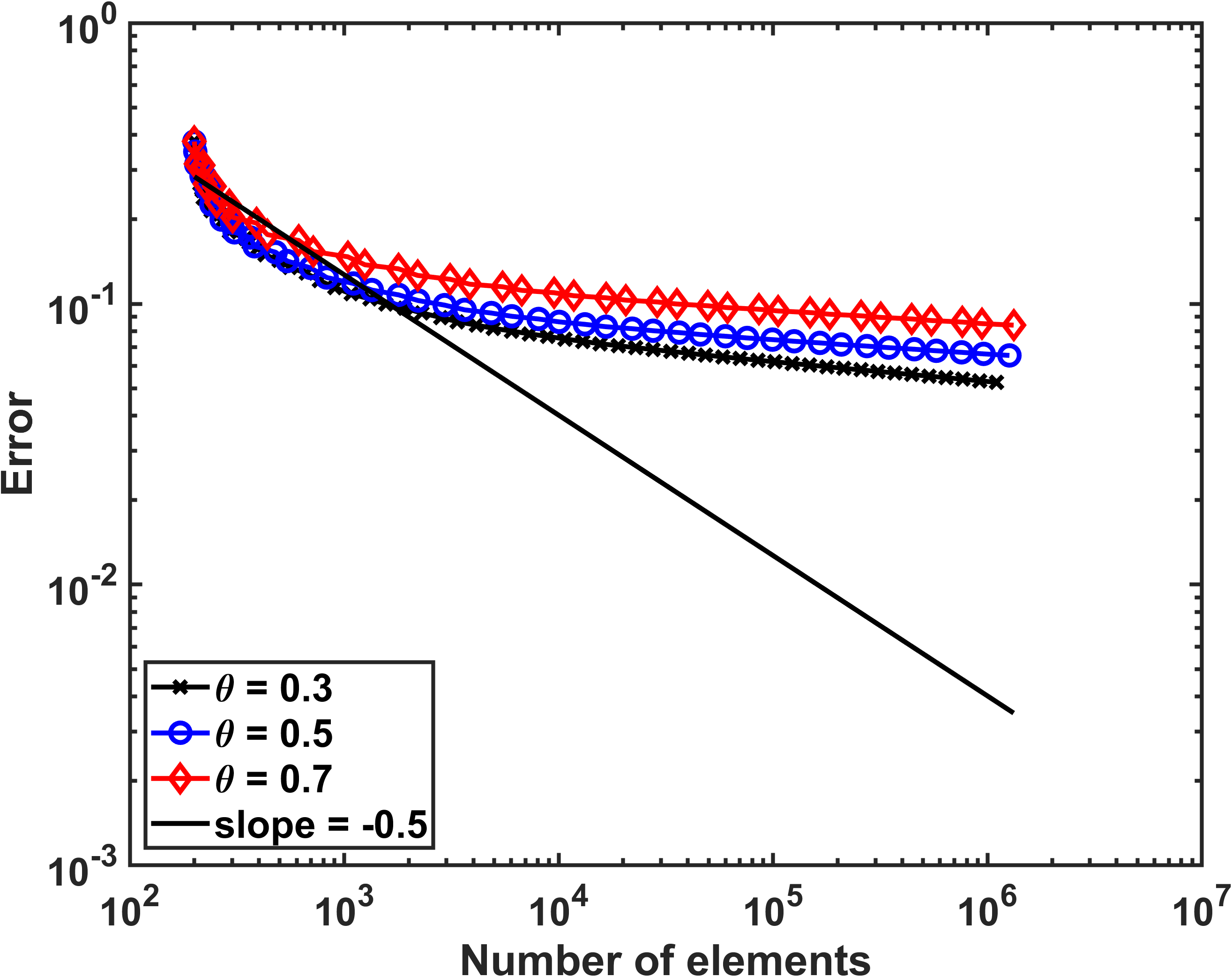}
	}
	\subfigure[CATGFEM]{
		\includegraphics[scale=0.32]{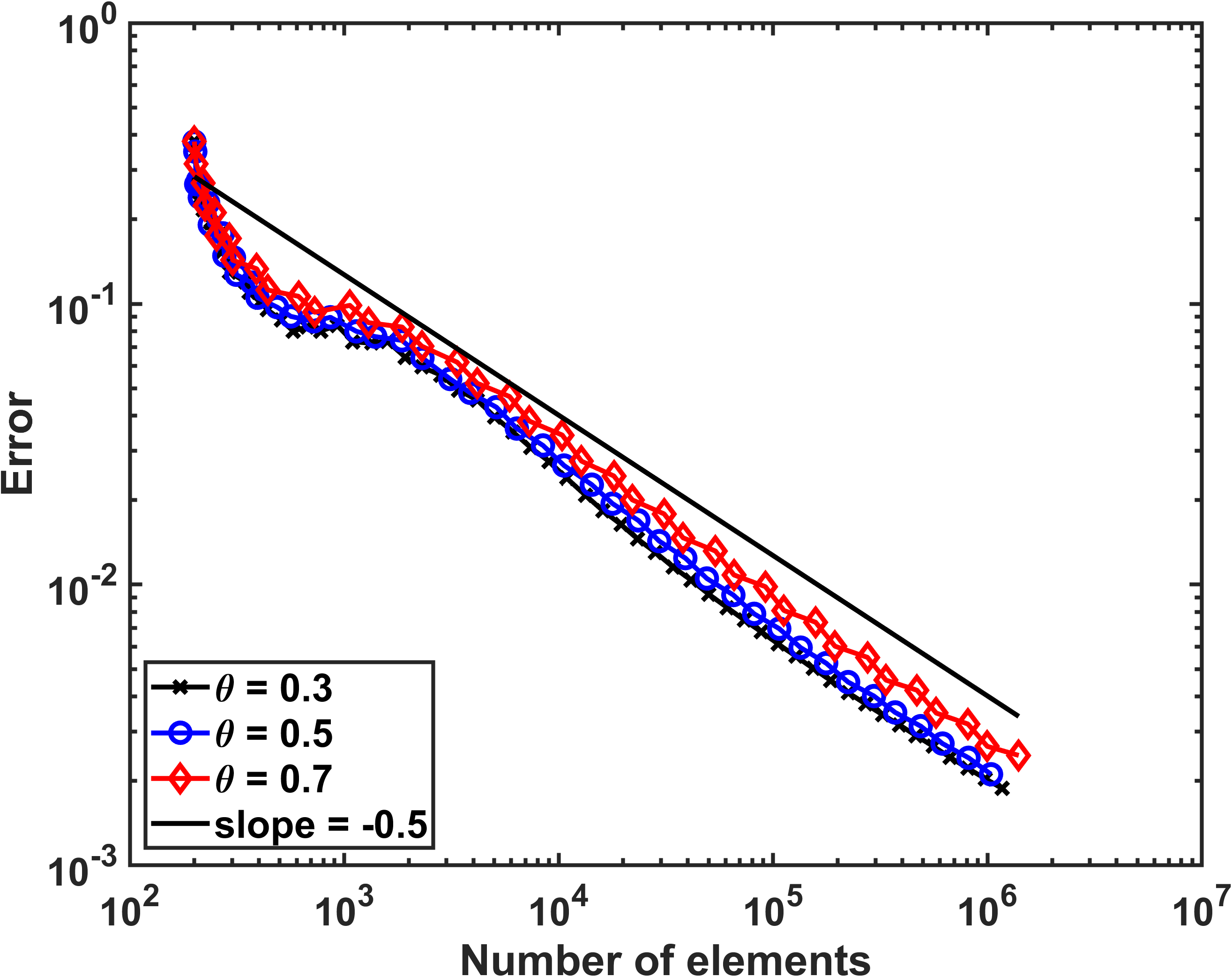}
	}
	\caption{Convergence history on the energy-norm errors of ATGFEM and CATGFEM with different marking parameters in Example \ref{test_peak-fracalpha}.}\label{fig_test_peak-fracalpha_rct--20_mark_algo-ATG-CATG_err}
\end{figure}

From Figure \ref{fig_test_peak-fracalpha_rct--20_mark_algo-ATG-CATG_err}-(a), we observe that the errors of the ATGFEM hardly change from some number of elements for all chosen marking parameters, while from Figure \ref{fig_test_peak-fracalpha_rct--20_mark_algo-ATG-CATG_err}-(b), we see that the errors of the CATGFEM are all optimally convergent, and that the error of the CATGFEM with $\theta = 0.7$ is a little larger than the identical ones with other values of $\theta$ at the same number of elements. Those phenomena indicate that the CATGFEM is more robust on the marking parameter than ATGFEM in this example.

The next example will investigate the behaviors of the ATGFEM and CATGFEM for the small diffusion coefficient $\valpha$. 

\begin{example}\label{test_Lshape-smalldfs}	
	We consider the L-shaped domain $\Omega = (-1, 1)\times (-1, 1) \backslash [0, 1)\times (-1, 0]$ in \eqref{model-1}, and select coefficients 
		\begin{equation*}
			\valpha = 0.1\mbI,
			\quad 
			\vbeta = \left[\begin{array}{c}
				r \\
				r
			\end{array}\right], 
			\quad
			\gamma = -1,   
		\end{equation*}
		and the source term $f$ determined by the following analytical solution 
		\begin{equation*}
			u = r^{2/3}\sin \frac{2\xi}{3}, \quad r = \sqrt{x^2 + y^2}, \quad 
			\xi = \arctan \frac{y}{x},
		\end{equation*}
One can note that this example is a variant of \cite[Example 6.1]{ChenHXXuXJ10:4492}.
\end{example}

We test the performance of the ATGFEM for a small marking parameter $\theta = 0.3$ and a very fine uniform initial mesh with $h_0 = 1/100$. Figure \ref{fig_Lshape_algo-compare_err}-(a) gives the corresponding result, which demonstrates that the energy-norm error of the ATGFEM is divergent from some number of elements, so we can assert that the ATGFEM does not work when the marking parameter $\theta$ is bigger than $0.3$ and the initial mesh size $h_0$ is larger than $1/100$ in this example.

\begin{figure}[!ht]
	\centering
	\subfigure[]{
		\includegraphics[scale=0.32]{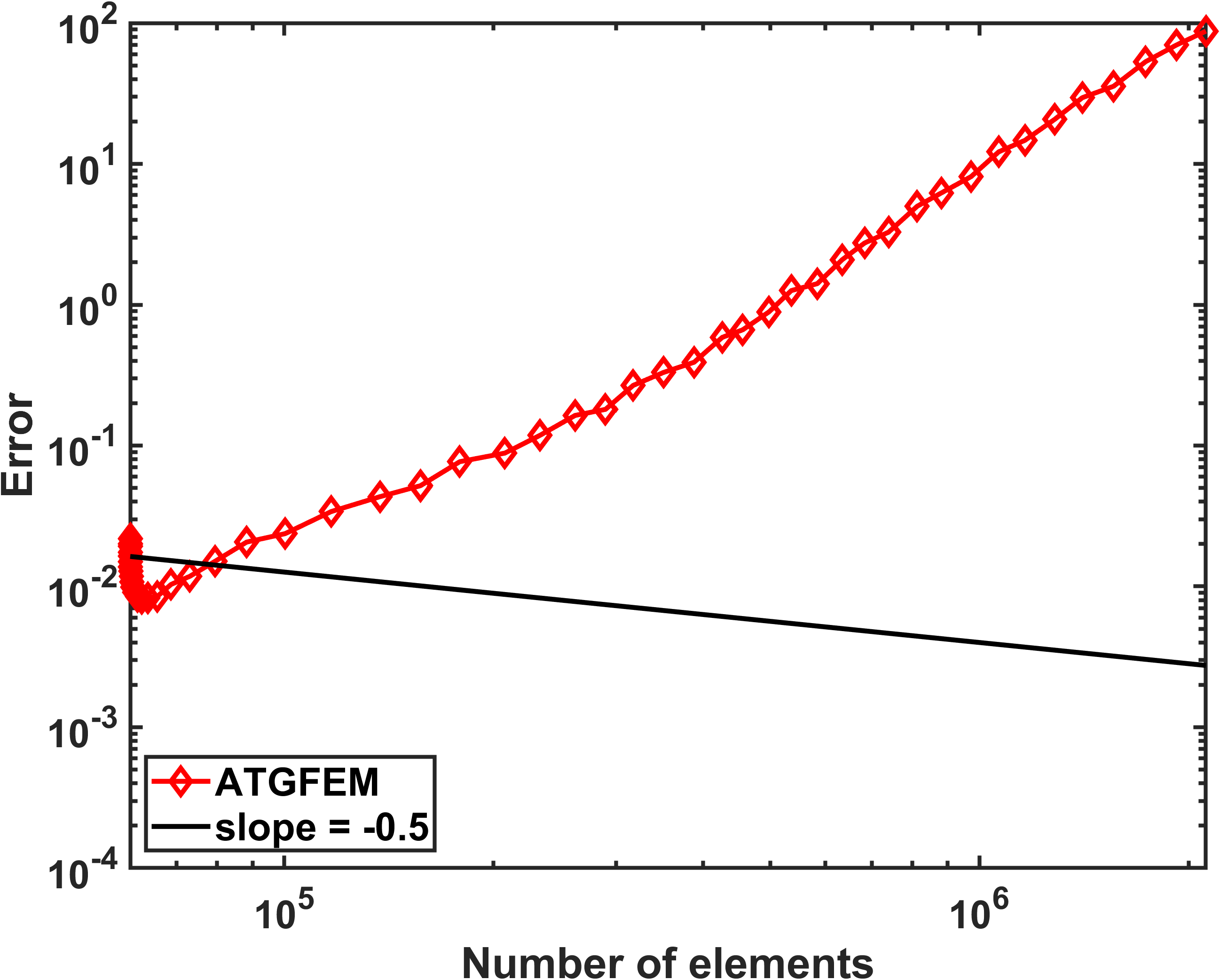}
	}
	\subfigure[]{
		\includegraphics[scale=0.32]{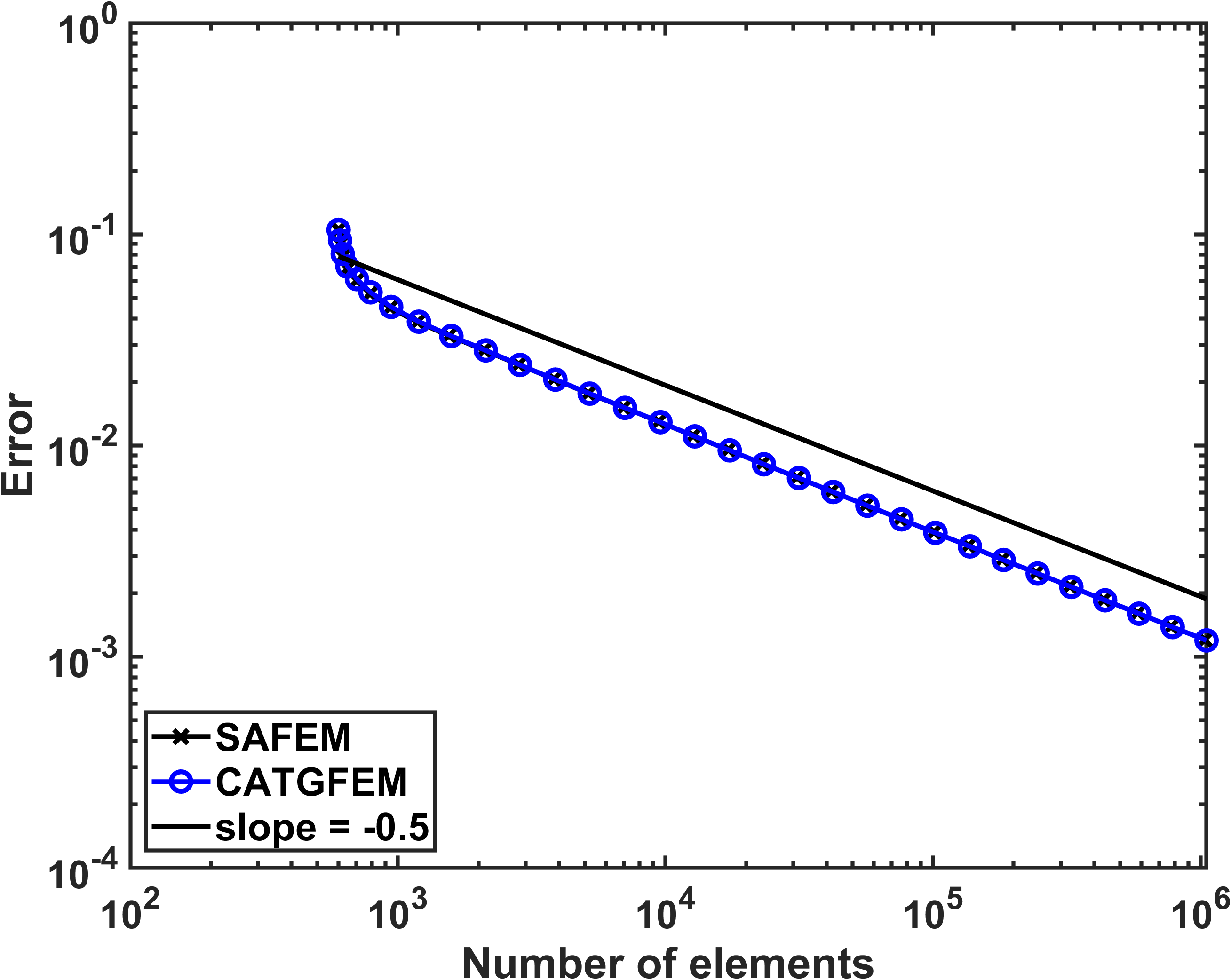}
	}
	\caption{Performances of the energy-norm errors of tested three adaptive algorithms in Example \ref{test_Lshape-smalldfs}: (a) ATGFEM with $\theta = 0.3$ and $h_0 = 1/100$; (b) SAFEM and CATGFEM with $\theta = 0.5$ and $h_0 = 1/10$.}\label{fig_Lshape_algo-compare_err}
\end{figure}

\begin{figure}[!ht]
	\centering
		\includegraphics[scale=0.45]{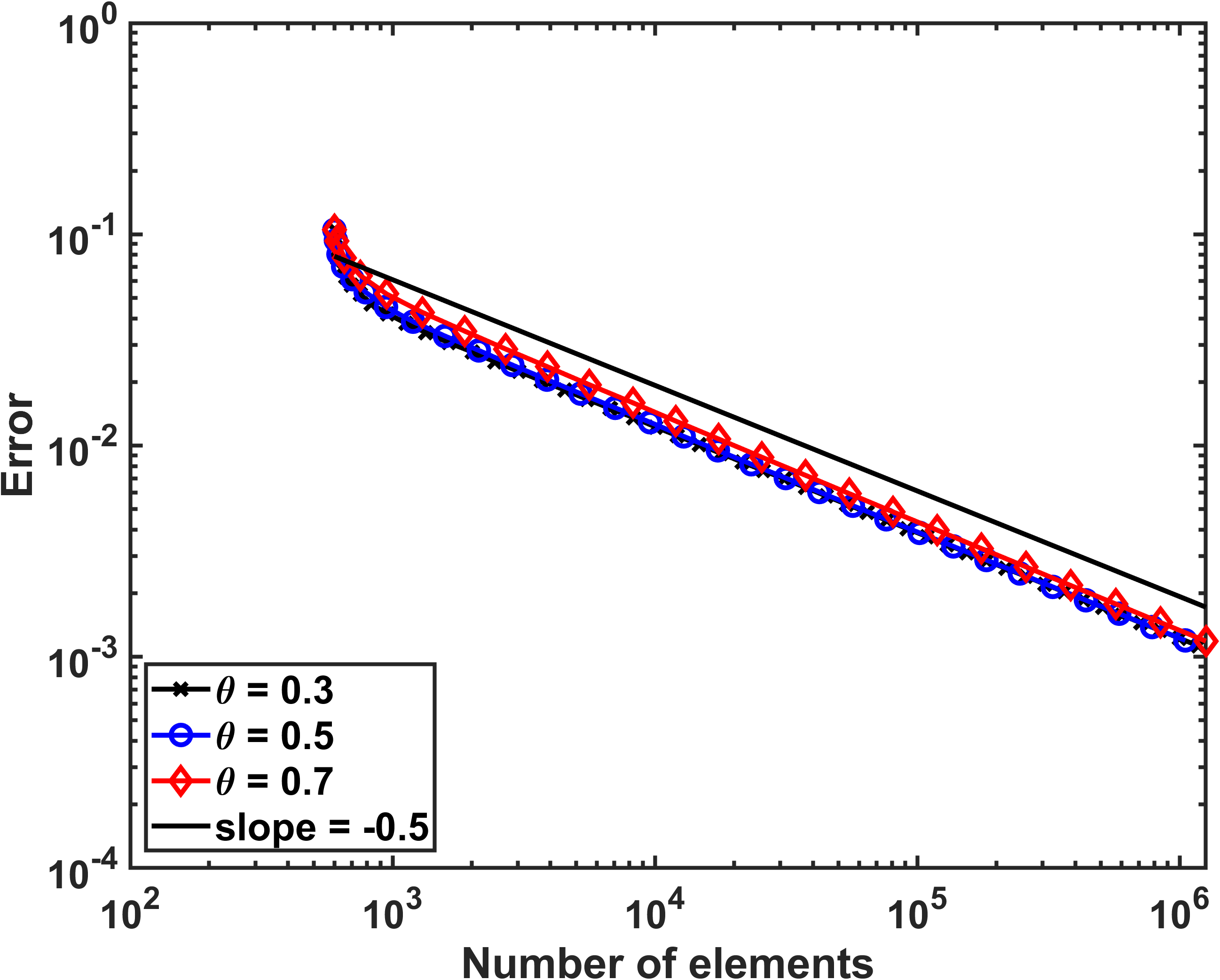}
	\caption{Convergence history on the energy-norm errors of CATGFEM with different marking parameters under $h_0 = 1/10$ in Example \ref{test_Lshape-smalldfs}.}\label{fig_Lshape_CATG_mark-compare_err}
\end{figure}

In contrast, the behaviors of the energy-norm errors of the SAFEM and the CATGFEM for a relatively big marking parameter $\theta = 0.5$ and a relatively coarse uniform initial mesh with $h_0 = 1/10$ are shown in Figure \ref{fig_Lshape_algo-compare_err}-(b), from which we see that the CATGFEM behaves well like the SAFEM, namely, their errors are both identical and converge at the optimal rate.

We next show the convergence performances of the energy-norm errors of the CATGFEM under $h_0 = 1/10$ with different marking parameters $\theta$ in Figure \ref{fig_Lshape_CATG_mark-compare_err}, which indicates that the CATGFEM has the great robustness on the marking parameter $\theta$ since its errors being nearly identical for all values of $\theta$ exhibit the optimal convergence.

The final example will test behaviors of the ATGFEM and the CATGFEM for a convection-dominated convection-diffusion problem. 

	\begin{example}\label{test_cvt-domin}	
	As in \cite[Example 6.2]{ChenHXXuXJ10:4492}, we select the unit square domain $\Omega = (0,1)^2$, the coefficients and the source term 
		\begin{equation*}
			\valpha =0.006\mbI,
			\quad 
			\vbeta = \left[\begin{array}{c}
				y \\
				0.6-x
			\end{array}\right], 
			\quad
			\gamma = 0, \quad\text{and}\quad f = 0,   
		\end{equation*}
in \eqref{model-1}, where the analytical solution is unknown but satisfies the following Dirichlet boundary condition
$$
u(x, y)= \begin{cases}1, & [0.3+\tau,\;0.6-\tau]\times \{0\}, \\ 
	0, & \partial \Omega \backslash [0.3,\;0.6] \times \{0\}, \\ 
	\text { linear, } & [0.3,\;0.3+\tau]\times\{0\} \text { or }[0.6-\tau, \; 0.6] \times \{0\}, \end{cases}
$$
with a parameter $\tau = 0.003$.
\end{example}

We choose a very fine uniform initial mesh with $h_0 = 1/256$ and a small marking parameter $\theta = 0.3$ in Example \ref{test_cvt-domin} to test the behavior of the ATGFEM. Figure \ref{fig_cvt-domin_algos}-(a) shows the corresponding result. from which we observe that the energy-norm error of the ATGFEM is also increasing from some number of elements like Figure \ref{fig_Lshape_algo-compare_err}-(a). Therefore, the ATGFEM does not work if the initial mesh size $h_0 \geq 1/256$ and the marking parameter $\theta \geq 0.3$. 

\begin{figure}[!ht]
	\centering
	\subfigure[]{
		\includegraphics[scale=0.32]{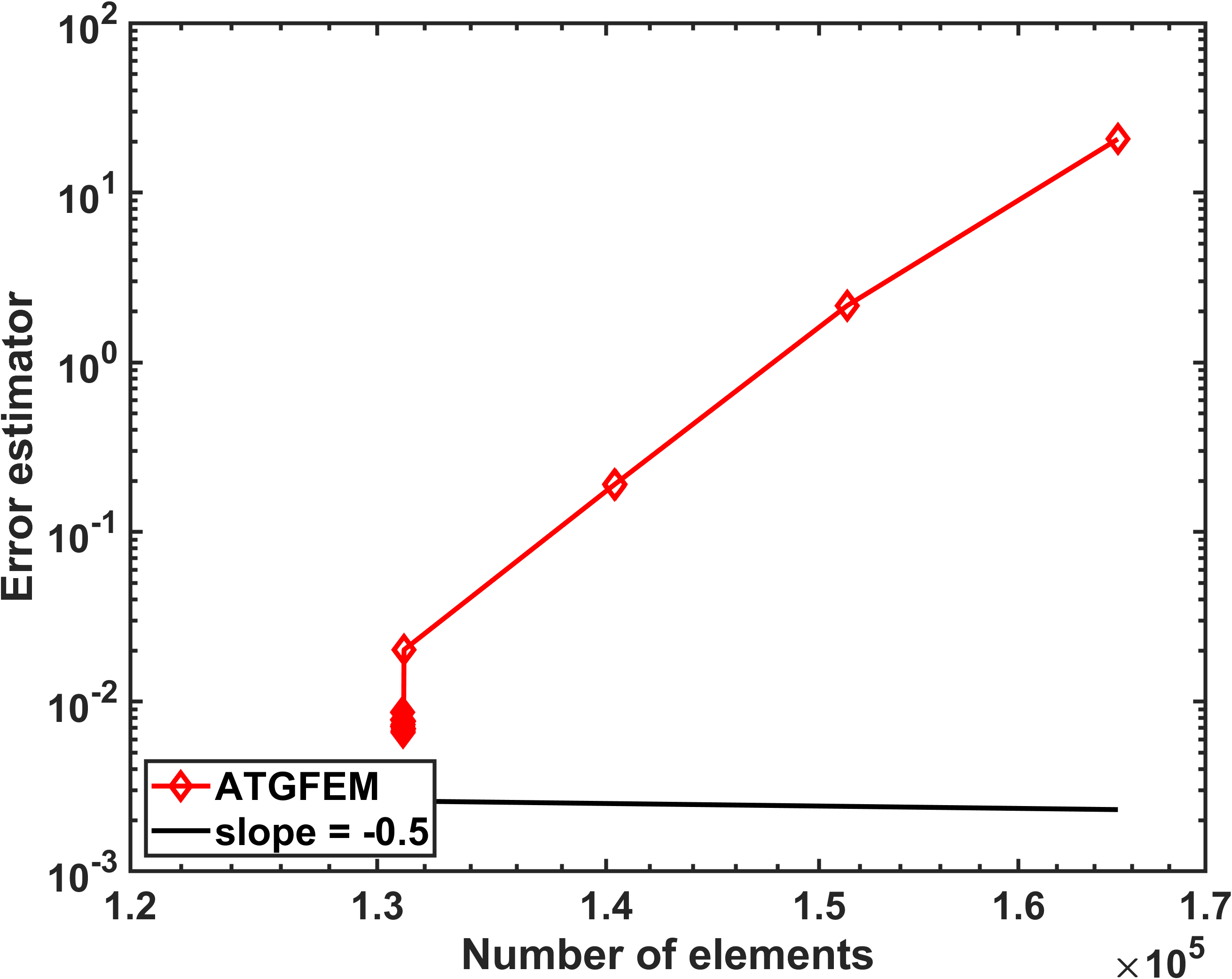}
	}
	\subfigure[]{
		\includegraphics[scale=0.32]{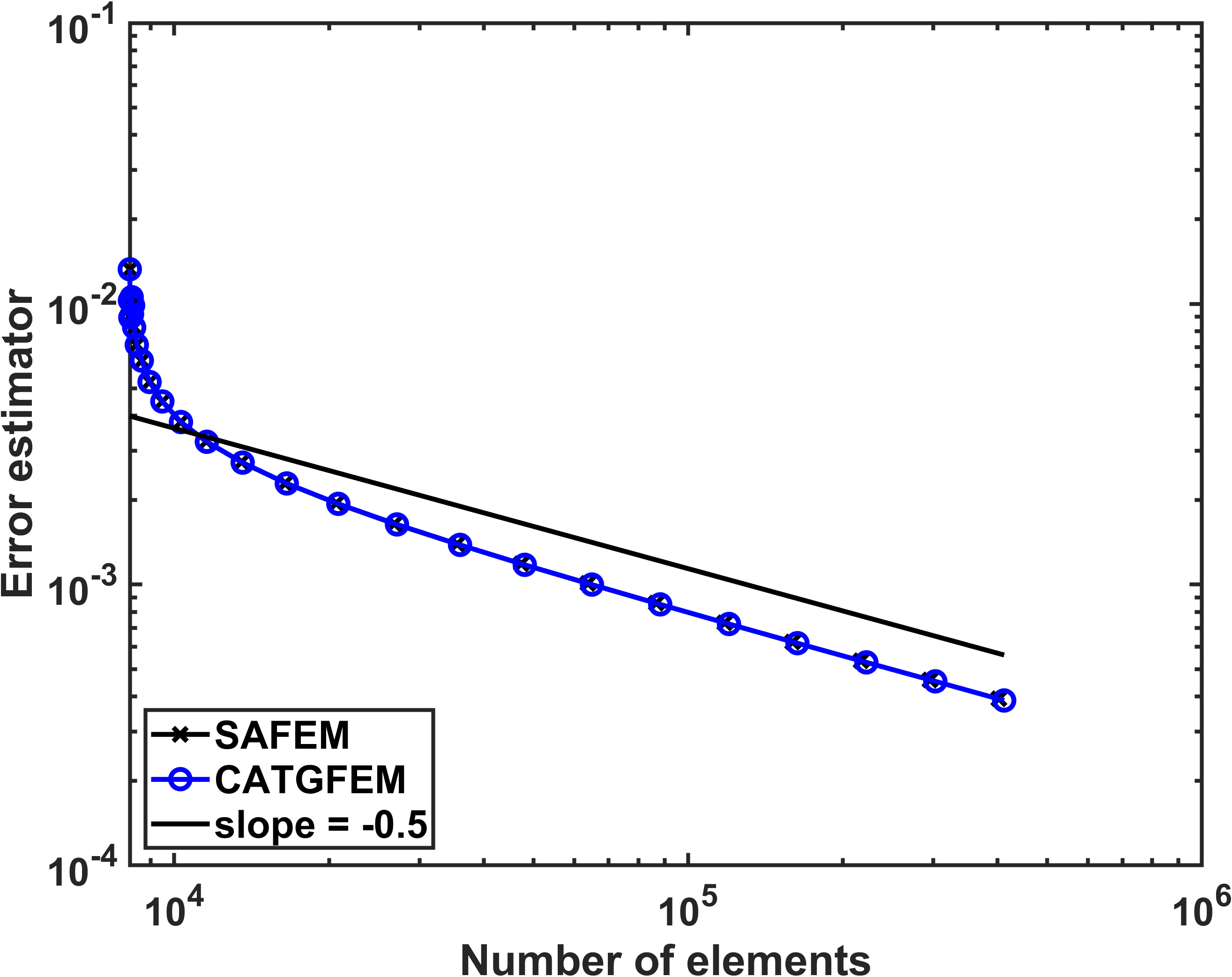}
	}
	\caption{Behaviors of the energy-norm errors of tested three adaptive algorithms in Example \ref{test_cvt-domin}: (a) ATGFEM with $h_0 = 1/256$ and $\theta = 0.3$; (b) SAFEM and CATGFEM with $h_0 = 1/64$ and $\theta = 0.5$.}\label{fig_cvt-domin_algos}
\end{figure}

In contrast, the behaviors of the energy-norm errors of the SAFEM and the CATGFEM are suggested in Figure \ref{fig_cvt-domin_algos}-(b) when choosing a relatively coarse uniform initial mesh with $h_0 = 1/64$ and a relatively big marking parameter $\theta = 0.5$, we see from Figure \ref{fig_cvt-domin_algos}-(b) that the CATGFEM also performs well like the SAFEM, as seen in Figure \ref{fig_Lshape_algo-compare_err}-(b). 

We then pictures the convergence behaviors of the energy-norm errors of the CATGFEM under $h_0 = 1/64$ with different marking parameters in Figure \ref{fig_cvt-domin_CATG}, which illustrates that the CATGFEM also has the excellent robustness on the marking parameter for the convection-dominated problem of this example. 

\begin{figure}[!ht]
	\centering
		\includegraphics[scale=0.45]{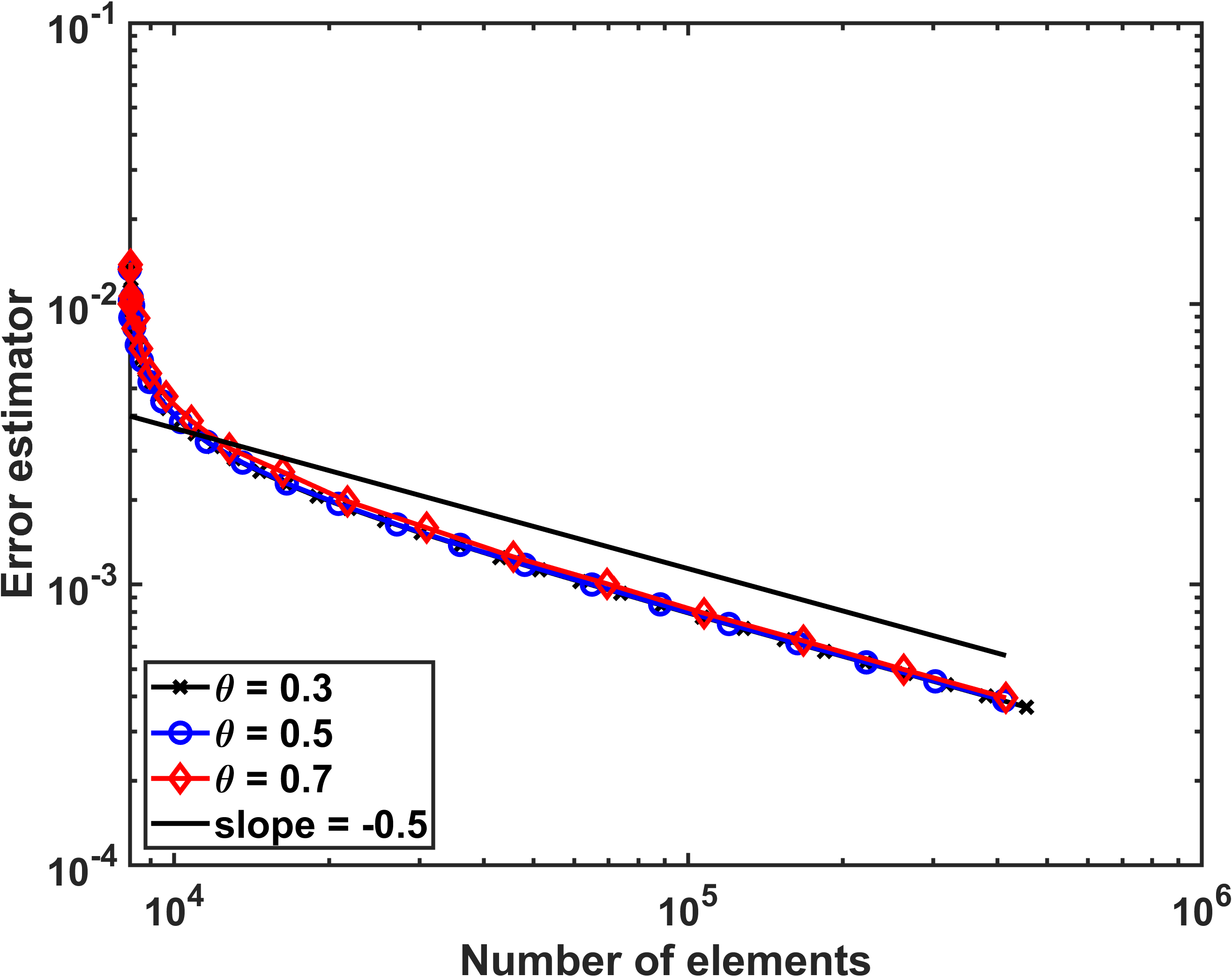}
	\caption{Convergence history on the energy-norm errors of CATGFEM with different marking parameters under $h_0 = 1/64$ in Example \ref{test_cvt-domin}.}\label{fig_cvt-domin_CATG}
\end{figure}

\section*{Acknowledgments}
The first author is supported by Scientific Research Innovation Project of Lingnan Normal University (LT2401) and High-Level Talents' Special Project of Lingnan Normal University (ZL2506). The third author is supported by the National Natural Science Foundation of China (12071160).

\bibliographystyle{abbrv}
\bibliography{Math}

@Article{ZhongLQXuanY22:113903,
  author       = {Zhong, Liu Qiang and Xuan, Yue and Cui, Jin Tao},
  journal      = {J. Comput. Appl. Math.},
  title        = {Two-grid discontinuous {G}alerkin method for convection--diffusion--reaction equations},
  year         = {2022},
  pages        = {113903},
  volume       = {404},
  creationdate = {2024-09-24T08:46:43},
}

@Article{AbdulleSouza22:110894,
  author       = {Abdulle, Assyr and Rosilho de Souza, Giacomo},
  journal      = {J. Comput. Phys.},
  title        = {A local adaptive discontinuous {G}alerkin method for convection-diffusion-reaction equations},
  year         = {2022},
  pages        = {110894},
  volume       = {451},
  creationdate = {2024-09-24T09:59:47},
}

@Article{DuanSYWuHJ23:21,
  author       = {Duan, Song Yao and Wu, Hai Jun},
  journal      = {J. Sci. Comput.},
  title        = {Adaptive {FEM} for Helmholtz equation with large wavenumber},
  year         = {2023},
  number       = {1},
  pages        = {21},
  volume       = {94},
  creationdate = {2024-09-24T10:04:11},
}

@Article{IrimieBouillard01:4027,
  author       = {Irimie, Simona and Bouillard, Ph},
  journal      = {Comput. Methods Appl. Mech. Engrg.},
  title        = {A residual a posteriori error estimator for the finite element solution of the Helmholtz equation},
  year         = {2001},
  number       = {31},
  pages        = {4027--4042},
  volume       = {190},
  creationdate = {2024-09-24T10:10:40},
}

@Article{DuYWuHJ15:782,
  author       = {Du, Yu and Wu, Hai Jun},
  journal      = {SIAM J. Numer. Anal.},
  title        = {Preasymptotic error analysis of higher order {FEM} and {CIP}-{FEM} for Helmholtz equation with high wave number},
  year         = {2015},
  number       = {2},
  pages        = {782--804},
  volume       = {53},
  creationdate = {2024-09-24T10:21:03},
}

@Book{Grisvard92Book,
  author       = {Grisvard, Pierre},
  title        = {Singularities in boundary value problems},
  year         = {1992},
  publisher    = {Springer-Verlag, Berlin},
  creationdate = {2024-09-24T10:30:03},
}

@Article{Mitchell13:350,
  author       = {Mitchell, William F},
  journal      = {Appl. Math. Comput.},
  title        = {A collection of 2{D} elliptic problems for testing adaptive grid refinement algorithms},
  year         = {2013},
  pages        = {350--364},
  volume       = {220},
  creationdate = {2024-09-24T10:32:43},
}

@Book{verfurth13Book,
  author       = {Verf{\"u}rth, R{\"u}diger},
  publisher    = {Oxford University Press, Oxford},
  title        = {A posteriori error estimation techniques for finite element methods},
  year         = {2013},
  creationdate = {2024-09-24T10:37:06},
}

@Article{BespalovHaberl17:318,
  author       = {Bespalov, Alex and Haberl, Alexander and Praetorius, Dirk},
  journal      = {Comput. Methods Appl. Mech. Engrg.},
  title        = {Adaptive {FEM} with coarse initial mesh guarantees optimal convergence rates for compactly perturbed elliptic problems},
  year         = {2017},
  pages        = {318--340},
  volume       = {317},
  creationdate = {2024-09-24T10:38:43},
}

@Article{FeischlFuhrer14:601,
  author       = {Feischl, M and F{\"u}hrer, T and Praetorius, D},
  journal      = {SIAM J. Numer. Anal.},
  title        = {Adaptive {FEM} with optimal convergence rates for a certain class of non-symmetric and possibly non-linear problems},
  year         = {2014},
  number       = {2},
  pages        = {601-625},
  volume       = {52},
  creationdate = {2024-09-24T10:51:01},
  doi          = {10.1137/120897225},
}

@Article{LiYKZhangY21:A908,
  author       = {Li, Yu Kun and Zhang, Yi},
  journal      = {SIAM J. Sci. Comput.},
  title        = {Analysis of adaptive two-grid finite element algorithms for linear and nonlinear problems},
  year         = {2021},
  number       = {2},
  pages        = {A908--A928},
  volume       = {43},
  creationdate = {2024-09-24T10:53:51},
}

@Book{AinsworthOden00Book,
  title     = {A posteriori error estimation in finite element analysis},
  publisher = {Wiley-Interscience, New York},
  year      = {2000},
  author    = {Ainsworth, Mark and Oden, J Tinsley},
}

@ARTICLE{ChenHXXuXJ10:4492,
  author = {Chen, Huang Xin and Xu, Xue Jun},
  title = {Local multilevel methods for adaptive finite element methods for
	nonsymmetric and indefinite elliptic boundary value problems},
  journal = {SIAM J. Numer. Anal.},
  year = {2010},
  volume = {47},
  pages = {4492-4516},
  number = {6}
}

@BOOK{Ciarlet02Book,
  title = {The finite element method for elliptic problems},
  publisher = {SIAM, Philadelphia},
  year = {2002},
  author = {Ciarlet, Philippe G}
}

@ARTICLE{DuSHXieXP15:1327,
  author = {Du, Shao Hong and Xie, Xiao Ping},
  title = {Convergence of an adaptive mixed finite element method for convection-diffusion-reaction
	equations},
  journal = {Sci. China Math.},
  year = {2015},
  volume = {58},
  pages = {1327-1348},
  number = {6}
}

@ARTICLE{Dorfler96:1106,
  author = {D{\"o}rfler, Willy},
  title = {A convergent adaptive algorithm for {Poisson}'s equation},
  journal = {SIAM J. Numer. Anal.},
  year = {1996},
  volume = {33},
  pages = {1106-1124},
  number = {3}
}

@BOOK{McCormick87Book,
  title = {Multigrid methods},
  publisher = {SIAM, Philadelphia},
  year = {1987},
  author = {McCormick, Stephen F.}
}

@ARTICLE{MekchayNochetto05:1803,
  author = {Mekchay, Khamron and Nochetto, Ricardo H},
  title = {Convergence of adaptive finite element methods for general second order linear elliptic {PDEs}},
  journal = {SIAM J. Numer. Anal.},
  year = {2005},
  volume = {43},
  pages = {1803-1827},
  number = {5},
  url = {http://www.jstor.org/stable/4101295}
}

@INBOOK{NochettoSiebert09:409,
  title = {Theory of adaptive finite element methods: an introduction},
  publisher = {Springer, Berlin},
  year = {2009},
  author = {Nochetto, Ricardo H. and Siebert, Kunibert G. and Veeser,Andrea},
  booktitle = {Multiscale, Nonlinear and Adaptive Approximation}
}

@ARTICLE{SchatzWang96:19,
  author = {Schatz, Alfred H and Wang, Jun Ping},
  title = {Some new error estimates for {Ritz-Galerkin} methods with minimal regularity assumptions},
  journal = {Math. Comp.},
  year = {1996},
  volume = {65},
  pages = {19-27},
  number = {213}
}

@ARTICLE{Schatz74:959,
  author = {Schatz, Alfred H},
  title = {An observation concerning {Ritz-Galerkin} methods with indefinite bilinear forms},
  journal = {Math. Comp},
  year = {1974},
  volume = {28},
  pages = {959-962},
  number = {128}
}

@ARTICLE{ScottZhang90:483,
  author = {Scott, L Ridgway and Zhang, Shang You},
  title = {Finite element interpolation of nonsmooth functions satisfying boundary conditions},
  journal = {Math. Comp.},
  year = {1990},
  volume = {54},
  pages = {483-493},
  number = {190},
  url = {http://www.jstor.org/stable/2008497}
}

@ARTICLE{Stevenson08:227,
  author = {Stevenson, Rob},
  title = {The completion of locally refined simplicial partitions created by bisection},
  journal = {Math. Comp.},
  year = {2008},
  volume = {77},
  pages = {227-241},
  number = {261}
}

@BOOK{Verfurth96Book,
  title = {A review of a posteriori error estimation and adaptive mesh-refinement techniques},
  publisher = {Wiley-Teubner, Chichester-Stuttgart},
  year = {1996},
  author = {Verf{\"u}rth, R{\"u}diger}
}

@Article{XuJC96:1759,
  author    = {Xu, Jin Chao},
  journal   = {SIAM J. Numer. Anal.},
  pages     = {1759-1777},
  title     = {Two-grid discretization techniques for linear and nonlinear {PDEs}},
  volume    = {33},
  year      = {1996},
  doi       = {10.1137/S0036142992232949},
  number    = {5},
  url       = {http://epubs.siam.org/doi/abs/10.1137/S0036142992232949},
}

@Article{XuJC92:303,
  author           = {Xu, Jin Chao},
  journal          = {SIAM J. Numer. Anal.},
  title            = {A new class of iterative methods for nonselfadjoint or indefinite problems},
  year             = {1992},
  number           = {2},
  pages            = {303-319},
  volume           = {29},
  doi              = {10.1137/0729020},
  modificationdate = {2023-08-29T22:22:26},
  url              = {http://epubs.siam.org/doi/abs/10.1137/0729020},
}

@Article{BiCJWangC18:23,
  author    = {Bi, Chun Jia and Wang, Cheng and Lin, Yan Ping},
  journal   = {J. Sci. Comput.},
  pages     = {23-48},
  title     = {A posteriori error estimates of two-grid finite element methods for nonlinear elliptic problems},
  volume    = {74},
  year      = {2018},
  number    = {1},
}

\end{document}